\documentclass[11pt]{article}
\usepackage{amsmath,amsthm,amssymb}

\theoremstyle{plain}
\newtheorem{thm}{Theorem}[section]
\newtheorem{lem}[thm]{Lemma}
\newtheorem{prop}[thm]{Proposition}
\newtheorem{cor}[thm]{Corollary}

\theoremstyle{definition}

\newtheorem{rem}[thm]{Remark}

\newcommand{\Z}{\mathbb Z}

\newcommand{\C}{\mathbb C}
\newcommand{\Se}{\mathcal S}
\newcommand{\cp}{\psi}
\newcommand{\ncp}{\tilde{\phi}}

\makeatletter
\def\imod#1{\allowbreak\mkern10mu({\operator@font mod}\,\,#1)}
\makeatother

\DeclareMathOperator{\ann}{Ann}
\DeclareMathOperator{\spa}{span}

\begin{document}

\title{Primitive Ideals of Noetherian Generalized
Down-Up Algebras}
\author{Iwan Praton}
\date{}

\maketitle

\begin{abstract}
We classify the primitive ideals of noetherian generalized down-up algebras.
\end{abstract}

\noindent\textit{Keywords}: down-up algebra; generalized down-up algebra; 
enveloping algebra; primitive ideal.\\
\textit{2000 Mathematics Subject Classification:} 16D60; 17B35; 16D25; 16S36.

\section{Introduction and preliminaries}
\subsection{Background}
We can construct $U(\mathfrak{sl}_2)$, the enveloping algebra of
$\mathfrak{sl}_2$,  as the algebra generated
by three elements $h, x, y$, subject to the relations
\[
hx-xh=x,\quad yh-hy=y, \quad xy-yx=h.
\]
(The first two relations are somewhat nontraditional.)
There are various ways of generalizing
this construction.
One could increase the number of generators;
this could lead to enveloping algebras of other
Lie algebras. Another path would be to
stay with three generators but modify
the defining relations. We would then
get an algebra with a reasonably small
dimension, making the algebra
computationally tractable.

One such example is Smith's algebras
similar to $U(\mathfrak{sl}_2)$; see~\cite{Smith}.
These algebras are generated by $h,x,y$ subject
to
\[
hx-xh=x,\quad yh-hy=y, \quad xy-yx=\phi(h),
\]
where $\phi$ is a polynomial. Thus
$h$ in this case is no longer the
commutator of $x$ and $y$, but
a polynomial root of said commutator.
The resulting algebras share many 
properties with
$U(\mathfrak{sl}_2)$. 

Another example, going in a different
direction, are the down-up algebras
of Benkart and Roby~\cite{BenkartRoby}.
These can be defined as the algebras
generated by $h, u, d$ subject to
the relations
\[
hu-ruh=\gamma u,\quad
dh-rhd=\gamma d,\quad
du-sud=h,
\]
where $r, s$, and $\gamma$ are constants.
When $\gamma\neq 0$,
these algebras have similar
defining relations with $U(\mathfrak{sl}_2)$,
except the commutators have
been modified. Down-up algebras,
especially when $rs\gamma\neq 0$,
also share many properties
with $U(\mathfrak{sl}_2)$.

Other generalizations of $U(\mathfrak{sl}_2)$
exist, some of which
can be considered as a mixture
of Smith's algebras and down-up algebras.
For example, Rueda has studied
algebras that are generated by
$h,x,y$ subject to
\[
hx-xh=x,\quad yh-hy=y,\quad xy-sxy=\phi(h),
\]
where $s$ is a constant and $\phi$
is a polynomial~\cite{Rueda}. For
another example, consider the conformal
$\mathfrak{sl}_2$ enveloping
algebras of Le Bruyn~\cite{LeBruyn}, 
which are generated
by $h,u,d$ subject to
\[
hu-ruh=u, \quad dh-rhd=d,\quad du-sud=ah^2+h,
\]
where $r,s,a$ are constants with $rs\neq 0$.

In 2004 Cassidy and Shelton introduced
the ultimate mixture of
Smith's algebras and down-up algebras%
~\cite{CassidyShelton}. These algebras
are generated by $h,u,d$ subject to
\[
hu-ruh=\gamma u, 
\quad dh-rhd= \gamma d,
\quad du-sud=\phi(h),
\]
where $r,s,\gamma$ are constants
and $\phi$ is a polynomial. It is
these algebras that are the subject
of this paper.

More specifically, 
we classify the primitive ideals of noetherian generalized
down-up algebras, hence completing the project begun in~\cite{Praton2}
and~\cite{Praton3}. As before, we try to provide a reasonably explicit
list of generators for these primitive ideals. Most of the necessary
techniques are straightforward generalizations of~\cite{Praton1}. 
In particular,
most of the time we will be using explicit computation---part of the
charm of studying down-up algebras is that 
elementary computations actually lead to useful results; it is not 
necessary to rely on heavy theoretical machinery.  

\subsection{Notations}
As usual, we denote the complex numbers by $\C$; $\C^\times$
will denote the nonzero complex numbers. 
If $z\in\C$, it is convenient
to define $o(z)$ to be the order of $z$ in the multiplicative
group $\C^\times$. Thus  $z$ is a primitive $n$th root
of unity if and only if $o(z)=n$. If $z$ is not a root of unity,
we set $o(z)=\infty$. By convention, when we write
$o(z)=n$ or $m$, we take $n$ or $m$ to be finite, i.e.,
$z$ is a root of unity. 

We also
need to use an ordering on the degree
of two-variable polynomials, so we
describe here the ordering that we 
will use. The
monomial $x^iy^j$ has degree $(i,j)$;
the degrees are ordered ``alphabetically
by last name'', i.e., $(i,j)>(i',j')$ iff
$j>j'$ or $j=j'$ and $i>i'$. The degree
of a polynomial is the highest degree of its
constituent monomials.

We write $\langle a,b,\dots,z\rangle$
to denote the (two-sided) ideal generated
by the elements $a,b,\dots,z$. 

\subsection{Definition of generalized down-up algebras}
We now state a careful definition of our object
of study.
A generalized down-up algebra is an algebra over $\C$ parametrized
by three complex numbers and a complex polynomial. Specifically,
the algebra $L(\phi, r,s,\gamma)$, where $r,s,\gamma\in\C$ and
$\phi\in\C[x]$, is the $\C$-algebra generated by three
generators $u$, $d$, and $h$,
subject to the relations
\[
hu-ruh=\gamma u,\quad
dh-rhd=\gamma d,\quad
du-sud=\phi(h).
\]
(We follow the convention in~\cite{Praton3},
which is slightly different from~\cite{CassidyShelton}.) 
We often just write $L$
for $L(\phi,r,s,\gamma)$ when the parameter values are
implicitly known. The algebra $L$ is
noetherian if and only if $rs\neq 0$. Primitive
ideals in the non-noetherian case
was described in~\cite{Praton3}, so in this paper we always assume
that $rs\neq 0$. This assumption  implies that $L$
is a domain~\cite[Proposition~2.5]{CassidyShelton}.

\subsection{Bases}
In order to do computations, we need a basis for $L$. 
The standard basis consists of the monomials
$\{u^ih^jd^k: i,j,k\geq 0\}$. Since we are assuming that
$L$ is a domain, we can arrange the $u$s,
$d$s, and $h$s in any order, i.e., the monomials
$\{u^id^kh^j\}$, $\{d^kh^ju^i\}$, $\{d^ku^ih^j\}$, $\{h^ju^id^k\}$,
and $\{h^jd^ku^i\}$ (where $i,j,k\geq 0$) are all bases of $L$~\cite[Theorem 2.1]{CassidyShelton}.

\subsection{Grading}\label{grading}
There is a useful grading on $L$ that results from declaring
that $u$ has degree $+1$, $d$ has degree $-1$, and $h$
has degree $0$. (Thus the monomial $u^ih^jd^k$ has degree
$i-k$.) Clearly $L_0$, the elements of degree 0, is
itself an algebra. It turns out to be  a
polynomial algebra on the two variables $h$ 
and $ud$~\cite[Proposition 4.1]{CassidyShelton}.
If $i>0$, then any element in $L_i$ can be written
(uniquely) as $u^if(h,ud)$, where $f\in\C[x,y]$.
Similarly, if $k<0$, then any element in $L_k$ can
be written (uniquely) as $g(h,ud)d^k$, where
$g\in\C[x,y]$. 

As in the commutative case, we say
that an element of $L_i$ is \emph{homogeneous of
degree $i$}. Any $x\in L$ can be written as a sum
of homogeneous elements:  $x=\sum_{i\in\Z} x_i$,
where $x_i\in L_i$  and only finitely many of the $x_i$s
are nonzero. We define the length $\ell(x)$ of $x$
to be the number of nonzero $x_i$s: 
$\ell(x)=\#\{i\in \Z: x_i\neq 0\}$. 
(Thus an element of length 1 is a nonzero
homogeneous element.)

\subsection{Isomorphisms}
Different values of the parameters $\phi$, $r$, $s$, and $\gamma$
can lead to isomorphic algebras. For example, $L(\phi,r,s,\gamma)$
and $L(\psi,r,s,c\gamma)$ are isomorphic, where $c\neq 0$ and $\psi(x)=\phi(cx)$.
(The isomorphism sends $u$ to $u'$, $d$ to $d'$, and $h$
to $ch'$.) Thus we can assume that either $\gamma=0$ or
$\gamma=1$ without loss of generality. Similarly, there is an
isomorphism between $L(\phi,r,s,\gamma)$ and
$L(c\phi, r,s,\gamma)$ via $u\mapsto cu'$, $d\mapsto d'$,
and $h\mapsto h'$. Thus we can assume that the polynomial $\phi$
is either monic or zero. We will, however, continue to use $\gamma$
and $\phi$ without additional assumptions, since assuming
$\gamma=1$ or $\phi$ monic does not significantly
ease our workload.  But there \emph{is}
an isomorphism between generalized down-up algebras that
we will exploit heavily; see~\cite[Proposition 1.7]{CarvalhoLopes}
\begin{lem}
Suppose $r\neq 1$. Then $L(\phi,r,s,\gamma)$ is isomorphic
to $L(\psi,r,s,0)$, where $\psi(x)=\phi(x-\gamma/(r-1))$.
\end{lem}
This lemma clarifies the role of $\gamma$: in most
cases, $\gamma$ is not necessary! The information
carried by $\gamma$ can be transfered into the polynomial
$\phi$. Note that this phenomenon is not visible in the original
formulation of down-up algebras, since in that setting we do not 
have flexibility in the choice of $\phi$.

We will thus treat the cases $r=1$ and $r\neq 1$ differently.
When $r=1$, we will consider both $\gamma=0$ and
$\gamma\neq 0$. But when $r\neq 1$, we will assume
that $\gamma=0$; this cuts down the number of cases
we have to consider by almost half.

\subsection{Conformal algebras}
Here is an example of the usefulness of Lemma 1.1.
Recall that the algebra $L(\phi,r,s,\gamma)$ is called
\emph{conformal} if there exists a polynomial $\cp$
such that $s\cp(x)-\cp(rx+\gamma)=\phi(x)$. 
Conformal algebras are
nice because we can then define $H=ud+\cp(h)$;
we shall see that having such an element is quite useful.
For one, any polynomial in $ud$ and $h$
can be written as a polynomial in $H$ and $h$;
the commutation relations involving $H$ are more
convenient than those involving $ud$. Specifically,
it is straightforward to show that
$Hu=suH$ and $dH=sHd$. 

To determine exactly when
an algebra is conformal is not a complete triviality,
mostly due to the presence of $\gamma$,
but Lemma 1.1 allows us to ignore $\gamma$
most of the time, so we can determine quickly
which of our algebras are conformal~\cite[Lemma 1.6,
Proposition 1.8]{CarvalhoLopes}.
\begin{lem}
Suppose $\phi(x)=\sum_{i=0}^n a_ix^i$. 
If $s\neq r^j$ for $0\leq j\leq n$, then
$L(\phi,r,s,0)$ is conformal. If $s=r^j$ for
some $0\leq j\leq n$ and $a_j=0$, then
$L(\phi,r,s,0)$ is also conformal.
Otherwise $L(\phi,r,s,0)$ is not conformal.
Additionally, when $\gamma\neq 0$,
$L(\phi,1,s,\gamma)$ is conformal.
\end{lem}
Since conformal algebras behave somewhat
differently than nonconformal ones, we will
treat these two cases separately. It turns out
that the nonconformal case, even when
$\gamma=0$, has a similar flavor to
the situation when $\gamma\neq 0$.

\subsection{Schur's Lemma}
Finally, one of our main weapons is the
following result, well-known to representation
theorists as Dixmier's version of 
Schur's lemma~\cite[2.6.5]{Dixmier}.
\begin{lem}
Let $A$ be an $\C$-algebra and $M$ a simple
$A$-module whose dimension is
countable. If $\xi\in \hom_{\C}(M,M)$
commutes with the action of $A$ on $M$,
then $\xi$ acts as a scalar on $M$. 
\end{lem}
In particular, the center of $A$ acts
as scalar operators on $M$. We will
use this result so often that we will
not mention it explicitly.

\section{Weight modules and finite dimensional simple modules}
In this section we describe weight modules;
for our purposes, we are especially interested in 
universal weight
modules and finite-dimensional modules
(which are instances of weight modules).
When simple, these modules provide
almost all of our primitive ideals. In
some cases, we do have to
consider modules that are not weight
modules, but such cases are exceptional,
and we will deal with them when they
arise.

Finite dimensional simple modules have been 
classified in~\cite[Section 4]{CassidyShelton}, so
all we have to do here is figure out their annihilators.

\subsection{Universal weight Modules}\label{universal}
We first recall the definition of weight modules.
If $M$ is an $L$-module, then $v\in M$
is said to have weight $(\lambda,\beta)\in\C^2$
if $h\cdot v =\lambda v$ and $(ud)\cdot v = \beta v$.
The \emph{weight space} $M_{(\lambda,\beta)}$ 
is the linear space consisting of
all elements of $M$ with weight $(\lambda,\beta)$.
The module $M$ is a \emph{weight module} if
it is the (direct) sum of its weight spaces.  

There is a nice relation between weight modules
and the grading of $L$ described in section~\ref{grading}.
We need to recall the (invertible) operation on weights 
given by $\Phi: (\lambda,\beta)\mapsto
(r\lambda+\gamma, s\beta+\phi(\lambda))$. 
(See~\cite[section 4]{CassidyShelton}.) Then
a simple calculation shows that if $M$
is a weight module, then $L_iM_{(\lambda,\beta)}
\subseteq M_{\Phi^i(\lambda,\beta)}$
($i\in \Z$), i.e., elements of degree $i$
transform vectors of weight $(\lambda,\beta)$
into vectors whose weights is $i$ steps
away. 

There exist universal weight modules,
which we now describe. Let $(\lambda,\beta)$
be an arbitrary weight, and define
$(\lambda_i,\beta_i)=\Phi^i(\lambda,\beta)$
($i\in\Z$) as above. Then the universal
weight module $W(\lambda,\beta)$ is
the module with basis $\{v_i: i\in\Z\}$ and
\begin{gather*}
hv_i=\lambda_iv_i, \quad i\in\Z;\\
uv_i=v_{i+1},\quad dv_{-i}=v_{-i-1},\quad i\geq 0;\\
uv_{-i}=\beta_{-i+1}v_{-i+1}, \quad dv_i=\beta_iv_{i-1},
\quad i>0.
\end{gather*}
Any weight module is a quotient of a universal
weight module. (It is possible to describe
universal weight modules more succinctly
as a tensor product, but since we want to
do explicit calculations on these modules,
it is better to have explicit formulas for the
action of $L$.)

We will pay a lot of attention to whether
$W(\lambda,\beta)$ is simple. A straightforward
result along these lines is as follows. If
the weights $(\lambda_i,\beta_i)$ are
all distinct, and $\beta_i\neq 0$ for all
$i\in\Z$, then $W(\lambda,\beta)$
is simple. 

\subsection{Finite-dimensional modules}
We  now turn to finite-dimensional modules.
All simple finite-dimensional modules are
weight modules, and hence quotients
of the universal weight modules. We need, however,
to have more detailed information.
It turns out that there are
two types of finite-dimensional simple modules:
those with highest (and lowest) weights, and
those that are cyclic.

The highest weight simple modules can
be described as follows. Start with
a weight $(\lambda, 0)$, and as before write
$(\lambda_i,\beta_i)$ for  $\Phi^{i}(\lambda,0)$.
(Thus $\lambda_0=\lambda$ 
and $\beta_0=0$ in this notation.) 
Suppose the weights $(\lambda_i,\beta_i)$
are all distinct, $\beta_{n+1}=0$, and
$\beta_i\neq 0$ for $1\leq i\leq n$.
Then there is a simple module of
dimension $n+1$, say with basis
 $\{v_0, v_1,\ldots,
v_n\}$.  Each $v_i$ is a vector
with weight $(\lambda_i,\beta_i)$. 
The action of $L$ on this module 
is
\begin{gather*}
hv_i=\lambda_i v_i,\quad 0\leq i\leq n;\\
dv_0=0, \quad dv_i=\beta_i v_{i-1}, \quad 1\leq i\leq n;\\
uv_n=0,\quad uv_i=v_{i+1}, \quad 0\leq i\leq n-1.
\end{gather*}
We denote this simple module by
$F_{\text{hw}}(\lambda)$, similar to the
notation in~\cite{CassidyShelton}.

There are two types of cyclic simple modules, denoted
by $F_c(\zeta,\rho)$ and $\overline{F}_c(\zeta,\rho)$
in~\cite[Definition 4.6]{CassidyShelton}. We describe
$F_c(\zeta,\rho)$ first. Start with a weight 
$(\lambda_0,\beta_0)$, and suppose its orbit
under $\Phi$
is finite; say the orbit has $m$ distinct values
$(\lambda_i,\beta_i)$, $0\leq i\leq m-1$. Assume
that $(\lambda_i,\beta_i)\neq (0,0)$ for any $i$. Denote
this set of weights by $\zeta$. Let $\rho$ be
a nonzero complex number. Then there
is a simple module $F_c(\zeta,\rho)$ of dimension $m$
with basis, say,  
 $\{v_0,v_1,\ldots, v_{m-1}\}$, 
where the basis vectors are indexed
by the cyclic group $\Z/m\Z$. Thus, for example,
$v_{m+1}$ is the same as $v_1$. Each $v_i$ is
vector of weight
 $(\lambda_i,\beta_i)$. Then action of
$L$ on $F_c(\zeta,\rho)$ is
\[
hv_i=\lambda_i v_i,\quad
uv_i=\rho v_{i+1},\quad
dv_i=\beta_i v_{i-1}.
\]

The modules $\overline{F}_c(\zeta,\rho)$ are similar.
The parameters and the weights are the same;
the only difference with $F_c(\zeta,\rho)$ is in
the action of $L$:
\[
hv_i=\lambda_i v_i,\quad
uv_i=\beta_{i+1} v_{i+1},\quad
dv_i=\rho v_{i-1}.
\]
There is usually a lot of overlap between $F_c(\zeta,\rho)$
and $\overline{F}_c(\zeta,\rho')$ as $\rho$ and $\rho'$
vary over $\C^\times$, but we need to consider both types of cyclic modules
since $F_c(\zeta,\rho)$ and $\overline{F}_c(\zeta,\rho')$
are nonisomorphic when $\prod_{i\in\Z/m\Z} \beta_i = 0$. 

\subsection{Annihilators of finite-dimensional simple  modules}

We now figure out the annihilators of these simple modules.
First we look at $F_{\text{hw}}(\lambda)$. In this case,
let $J_\lambda=\{f\in L_0=\C[h,ud] : f(\lambda_i,\beta_i)=0, 0\leq i\leq n\}$.
Then $J_\lambda$ is a classical polynomial ideal; it is
finitely generated, and for specific values of $(\lambda_i,\beta_i)$
we can figure out a list of its generators.
\begin{prop}
The annihilator of $F_{\text{hw}}(\lambda)$ is
$\langle u^{n+1}, d^{n+1}, J_\lambda\rangle$. 
\end{prop}
\begin{proof}
Write $I$ for the ideal $\langle u^{n+1}, d^{n+1}, J_\lambda\rangle$. 
It is straightforward to verify that $I$ annihilates $F_{\text{hw}}(\lambda)$;
we need to show the reverse inclusion
$\ann F_{\text{hw}}(\lambda)\subseteq I$. 

If $n=0$, then $F_{\text{hw}}(\lambda)$ is one
dimensional; $u$ and $d$ act as the zero
operator while $h$ acts as the scalar $\lambda$.
The ideal $J_\lambda$ is generated by $ud$ and
$h-\lambda$, so in this case, $I=\langle
u,d,h-\lambda\rangle$. It is clear that $I$
is the annihilator of $F_{\text{hw}}(\lambda)$.
Thus we assume that $n\geq 1$ from now on. 

Now suppose $y$ is an element of $L$ that annihilates
$F_{\text{hw}}(\lambda)$. Since $L_iv_j\subseteq \C v_{j+i}$,
it does no harm to assume that $y$ is homogeneous, say of degree
$-k<0$. (If $y$ is homogeneous of positive degree,
the proof is similar.)

We'll utilize the element $x=u+d^n$. Note that $xv_i=v_{i+1}$
for $0\leq i\leq n-1$, but $xv_n=\beta_1\cdots \beta_n v_0$;
thus $x^n v_i=\eta v_i$, where $\eta=\prod_{i=1}^n \beta_i$
is nonzero. So $x^n$ acts a nonzero scalar on $F_{\text{hw}}(\lambda)$.

Note also that $x^2=u^2+ud^n+d^nu+d^{2n}\in u^2 + L_{1-n} + I$;
by induction we can show that $x^i\in u^i + L_{i-1-n} + I$ for
$1\leq i\leq n$.

Recall that $y\in d^kL_0$. Thus $yx^k\in (u^k+L_{k-1-n}+I)
(d^kL_0)\in L_0+I$. In other words, $yx^k\in f(h,ud)+I$, where
$f$ is a polynomial in two variables. But $yx^k$ annihilates
everything in $F_{\text{hw}}(\lambda)$, 
so $f(h,ud)$ also annihilates everything,
and hence $f(h,ud)$ must be in $J_\lambda$. So $yx^k$ is actually
in $I$.

Thus modulo $I$, we have $0\equiv yx^k\equiv yx^kx^{n-k}
\equiv yx^n\equiv \eta y$. Since $\eta\neq 0$, we conclude
that $y\equiv 0$, which is what we want to show.
\end{proof}

We now take a look at the annihilators of $F_c(\zeta,\rho)$
and $\overline{F}_c(\zeta,\rho)$. We again
utilize the ideal $J_\zeta=\{f(h,ud): f(\lambda_i,\beta_i)=0,
i\in\Z/m\Z\}$ in $L_0$. Write $\eta=\prod_{i=0}^{m-1}
\beta_i$.
\begin{prop}
The annihilator of $F_c(\zeta,\rho)$ is 
$\langle u^m-\rho^m, d^m-\eta, J_\zeta\rangle$, and 
the annihilator of $\overline{F}_c(\zeta,\rho)$ is 
$\langle d^m-\rho^m, u^m-\eta, J_\zeta\rangle$.
\end{prop}
\begin{proof}
The proof is similar to the previous proof. As before,
it is clear that $I$ annihilates $F_c(\zeta,\rho)$. We
want to show that $\ann F_c(\zeta,\rho)\subseteq I$.

Write 
$L'_k = \sum_{i\equiv k \text{ mod } m} L_i $. Then 
$L'_k v_i\subseteq
\C v_{i+k}$. Suppose now that $y$ annihilates 
$F_c(\zeta,\rho)$. It does no harm to
assume that $y\in L'_k$ for some $k\in\Z$. Since
$u^{am+k}\in \rho^{am}u^k+I$ and $d^{bm+(m-k)}
\in \eta^b d^{m-k}$, we can even assume that
$y\in u^kL_0+d^{m-k}L_0$ where $0\leq k\leq m-1$. 

Now $yu^{m-k}$ also annihilates $F_c(\zeta,\rho)$,
and $yu^{m-k}\in L_m+L_0 \subseteq L_0+I$. 
Therefore we can write $yu^{m-k} = f(h,ud)+I$
where $f$ is a polynomial of two variables.
Since $yu^{m-k}$ annihilates everything,
we conclude that $f\in J_\zeta$, i.e., 
$yu^{m-k}\in I$.

Therefore $yu^{m-k} u^k$ is also in $I$, and
so modulo $I$, we have $0\equiv yu^m\equiv
y\rho^m$. Because $\rho\neq 0$, we conclude
that $y\equiv 0$, as required.

The proof for $\overline{F}_c(\zeta,\rho)$ is
similar.
\end{proof}

\subsection{Detecting finite-dimensionality}
It is also useful to be able to tell
whether a simple module is finite-dimensional
from partial information about its  annihilator.
\begin{lem}\label{findim}
Let $M$ be a simple module such
that (i) either $d^m$ or $u^m$ ($m\geq 1$) 
acts as a scalar on $M$; (ii) $M$
contains a weight vector. Then 
$M$ is finite-dimensional.
\end{lem}
\begin{proof}
We suppose that $d^m$ acts as the scalar
$\rho^m$. (The proof where $u^m$ acts as
a scalar is similar.) We first consider
the case where $\rho=0$.

In this case the first step is to show
that $M$ contains an element with
weight $(\lambda,0)$. Recall that
if $v$ is a weight vector, then 
$d^iv$ is also a weight vector,
for all $i\geq 0$. Since $d^mv=0$,
there must be an $j\geq 0$ such
that $d^jv\neq 0$ and $d^{j+1}v=0$.
Then $w_0=d^jv$ is an element
of weight $(\lambda,0)$ for some
$\lambda\in\C$.

Now define $w_i=u^iw_0$, $i\geq 0$.
It is straightforward to
check that $\spa \{w_i: i\geq 0\}$ is stable 
under $u$, $d$,
and $h$; since $M$ is simple,  
we have $M=\spa \{w_i: i\geq 0\}$. 

Note that $dw_i=\beta_iw_{i-1}$ for
some $\beta_i\in\C$. Thus
we have
$0=d^mw_{m}=\beta_{m}\beta_{m-1}\cdots
\beta_{1}w_0$. Since
$w_0\neq 0$, we have or $\beta_k=0$ for
some $1\leq k\leq m$. Then $M_0=\spa
\{w_{k+i}: i\geq 0\}$ is a submodule
of $M$, and hence must be $0$ or $M$.
If $M_0=0$, then $M$ is spanned by
$\{w_0,\cdots,w_{k-1}\}$ and so
is finite-dimensional. If $M_0=M$, then in particular
$w_0$ is a linear combination of the
elements $w_{k+i}$, $i\geq 0$. Say
$w_0=\sum_{i=1}^t c_iw_{n_i}$,
where $c_i\neq 0$ and we arrange
the indices so that $n_i>n_j$ whenever
$i>j$. Then $w_{n_t}$ is a linear
combination of $w_0$ and
the elements $w_{n_i}$, $i<t$.
Applying the operator $u$, 
we see that any $w_j$, with $j\geq n_t$,
is a linear combination of elements
$w_i$, with $i<n_t$. Thus $M$
is spanned by $\{w_i: i<n_t\}$,
and hence finite-dimensional.

We now consider the case where
$\rho\neq 0$. As before, suppose
$v\in M$ is a weight vector.  
Let $v_i=u^iv$  for
$0\leq i\leq m-1$. Each $v_i$
is still a weight vector;
in particular, $hv_i\in\C v_i$ for all values
of $i$. 
Since $v_0$ is an eigenvector of $ud$,
we check easily that $dv_i\in \C_{i-1}$
for $1\leq j\leq m-1$. Since $u^m$ acts
as $\rho^m$ on $M$, we get that
$uv_{m-1}=\rho^m v_0$.
Thus $dv_0=\rho^{-m} duv_{m-1}=
\rho^{-m} [sud+\phi(h)]v_{m-1}
\in\C v_{m-1}$. 
Therefore the span of 
$\{v_i: 0\leq i\leq m-1\}$
is stable under $L$, and hence
must be all of $M$. Thus $M$
is finite-dimensional.
\end{proof}
\begin{cor}\label{findimcor}
Let $M$ be a simple module, and suppose
for some $m\geq 1$, we have
$d^m\in\ann M$, $d^{m-1}\not\in \ann M$,
and $f(h)d^{m-1}M=0$ for some
nonconstant polynomial $f$. 
Then $M$ is finite-dimensional.
(The same conclusion holds if
$d$ is replaced by $u$.)

In particular, 
if $d^mu-s^mud^m\neq 0$, 
then $M$ is finite-dimensional.
(Similarly, if $u^m\in\ann M$, $u^{m-1}\not\in \ann M$,
and $du^m-s^mu^md\neq 0$, then $M$
is finite-dimensional.)
\end{cor}
\begin{proof}
Suppose  $f$ has degree $n\geq 1$.
Pick a nonzero element $v\in d^{m-1}M$,
and let  $W=\spa \{v, hv, h^2v,\ldots,
h^{n-1}v\}$. Then $W$ is stable
under $h$, and since $W$ is
finite-dimensional, $h$ 
has an eigenvector in $W$, say $w$.
Note that $dhv=(rh+\gamma)dv=0$;
similarly, $dh^iv=0$ for all $i\geq0$.
Thus $dw=0$, and so $w$ is
a weight vector. We can now apply
the lemma. 

In the particular case, 
we show easily by induction
that, for $k\geq 1$,
\[
d^ku-s^kud^k=f_k(h)d^{k-1}
\]
for some polynomial $f_k$. The hypothesis
implies that $f_m(h)\neq 0$. Note that
$f_m(h)$ cannot be a constant since that
would imply $f_m(h)d^{m-1}M=0$, i.e.,
$d^{m-1}M=0$. Thus $f_m(h)$ is a
nonconstant polynomial, and 
the conclusion follows.
\end{proof}
It is also useful occasionally to detect when
a simple module is one-dimensional.
\begin{lem}\label{onedim}
Let $M$ be a simple module. Suppose 
$ud$ and $h$ act as scalars on $M$. Then
$M$ is one-dimensional.
\end{lem}
\begin{proof}
Say $ud$ acts as the scalar $\beta$
and $h$ acts the scalar $\lambda$. 
Pick a nonzero $w_0\in M$ and 
for $i\geq 0$ define
$w_i=u^iw_0$ and $w_{-i}=d^{i}w_0$.
For all $i\in\Z$, we clearly have
$uv_i\in\C v_{i+1}$; also, since $du=sud+\phi(h)$, 
we get that $dv_i\in \C v_{i-1}$ for $i\in \Z$.
Thus $\spa\{w_i: i\in\Z\}$ is a submodule
of $M$ and hence must be all of $M$.

If $i<0$, we have $uw_i=udw_{i+1}=\beta w_{i+1}$,
so $duw_i=\beta w_i$. If $i\geq 0$, then
$uduw_i=\beta uw_i=\beta w_{i+1}$; but
we also know that $duw_i=\alpha w_i$
for some $\alpha\in\C$, so 
$uduw_i=\alpha w_{i+1}$. Thus $\alpha=\beta$
and we conclude that $duw_i=\beta w_i$.

Therefore $ud$ and $du$ act as the same
scalar on $M$. Thus the operators $u$,
$d$, and $h$ all commute with each other,
and hence $M$ is one-dimensional. 
\end{proof}

\section{When $r=1$ and $\gamma\neq 0$}
In this section we tackle the case where $\gamma\neq 0$. Recall
from Lemma 1.1 that in this case it suffices to assume 
that $r=1$. 
We then have $hu=u(h+\gamma)$ and $hd=d(h-\gamma)$.
It follows easily that $f(h)u=uf(h+\gamma)$ and $f(h)d=df(h-\gamma)$
where $f$ is any polynomial; also, $hu^i=u^i(h+i\gamma)$
and $hd^i=(h-i\gamma)d^i$ for $i\geq 0$. We can
summarize succinctly by saying that
$f(h)x_i=x_if(h+i\gamma)$, where $x_i\in L_i$. 

\subsection{Homogeneous elements in ideals}
We start with a result about homogeneous elements.
\begin{lem}\label{homogen}
Let $I$ be an ideal of $L$.
Suppose $x\in I$ and $x=\sum_{i\in\Z} x_i$ (where
$x_i\in L_i$) is the homogeneous decomposition
of $x$. Then $x_i\in I$ for all $i\in \Z$. 
\end{lem}
\begin{proof}
We use induction on the length of $x$. (Recall that 
the length of $x$ is the number of
nonzero $x_i$s.)  
The lemma is certainly true for 
elements of length 1.
Assume it is true for elements of
length $n-1$, and let
$
x=\sum_{i\in\Z} x_i
$,
($x_i\in L_i$) be an element of length $n$ in $I$.
Thus exactly $n$ of the $x_i$s is nonzero.
Pick an integer
$m$ so that $x_m\neq 0$. Now 
\[
hx=\sum_{i\in\Z} hx_i =\sum_{i\in\Z} x_i (h+i\gamma), 
\]
so 
\[
xh+m\gamma x -hx 
=\sum_{i\in\Z} (m-i)\gamma x_i.
\]
This is an element in $I$ of length $n-1$.
We apply the induction hypothesis and conclude that
$(m-i)\gamma x_i\in I$ for $i\in \Z$, and
thus $x_i\in I$ for $i\neq m$. Therefore
$y=\sum_{i\neq m} x_i$ is an element of $I$,
hence so is $x_m=x-y$. 
\end{proof}
Recall from Lemma 1.2 that the algebra
$L(\phi,1,s,\gamma)$ is conformal. Thus there exists
a polynomial $\cp\in\C[x]$ such that $s\cp(x)-\cp(x+\gamma)
=\phi(x)$. Recall also that we define $H$ as the 
element $ud+\cp(h)$;
then
$Hu=suH$ and $dH=sHd$. We conclude
that $f(h,H)u
=uf(h+\gamma,sH)$ and $f(h,H)d
=df(h-\gamma, s^{-1}H)$ for any
polynomial $f\in\C[x,y]$. 

Note that  if $M$ is
an $L$-module, then 
$HM$ is stable under $u$ and $d$
(and certainly under $h$), so $HM$ is a
submodule. Thus if $M$ is simple,
then either $HM=0$ or $HM=M$. We treat
these two cases separately. 

\subsection{The case $HM=0$}\label{HM=0}
Suppose first that $HM=0$.
Of course, $\ann M$ then contains $\langle H\rangle$,
but if $\cp$ is the zero polynomial, we can say more.
In this case, $du=sud$, so $uM$ and $dM$ are
submodules of $M$, and hence either $uM=0$
or $uM=M$; similarly, either $dM=0$ or $dM=M$.
It cannot be the case that both $uM=M$
and $dM=M$, for then $HM=udM=M$,
contradicting $HM=0$. Thus when $\cp$
is the zero polynomial (and $HM=0$),
$\ann M$ must contain either $u$ or $d$.
\begin{lem}\label{containspower}
Suppose $M$ is simple and $HM=0$. If 
$\ann M\supsetneq\langle H\rangle$, then 
$\ann M$ contains a power of $u$
or a power of $d$. Furthermore, if $\cp$ is the zero
polynomial and $\ann M\supsetneq
\langle d\rangle$, then $\ann M$ contains
a power of $u$, while  if $\ann M\supsetneq 
\langle u\rangle$, then $\ann M$ 
contains a power of $d$. 
\end{lem}
\begin{proof}
When $\cp$ is not the zero
polynomial, let $I=\langle H\rangle$,
but when $\cp$ is the zero polynomial,
let $I$ denote either $\langle u\rangle$
or $\langle d\rangle$.  Suppose that
$\ann M$ strictly contains $I$.
Choose an element $x\in\ann M$
that is not in $I$; by Lemma~\ref{homogen}
we can assume that $x$ is homogeneous,
say of degree $k$. If $I=\langle u\rangle$,
then $k\leq 0$;
if $I=\langle d\rangle$, then $k\geq 0$; 
if $I=\langle H\rangle$,
then there is no restriction on $k$.
To avoid repetitions, we assume
$k\geq 0$ here (and thus
$I=\langle d\rangle$ if $\cp=0$); 
the other possibilities
are treated similarly.
Then we can write $x=u^kg(h,H)$, but
since $x\not\in I$, we can assume that
$x=u^kf(h)$ for some polynomial $f$.

Among all nonzero
elements of the form $u^kf(h)$, pick one where
the degree of $f$ is as small as possible.
Then
\[
xu-ux=u^{k+1}(f(h+\gamma)-f(h)).
\]
The polynomial $f(h+\gamma)-f(h)$ has smaller 
degree than $f(h)$, hence it must be the zero
polynomial by choice of $x$. Thus we have
$f(h+\gamma)=f(h)$, and this implies that
$f$ is the constant polynomial. Thus $\ann M$
contains $u^k$, where $k\geq 1$.
\end{proof}
We can now derive some consequences
about $M$.
\begin{lem}\label{bigideals2}
Suppose $M$ is simple and $HM=0$. If $\cp$
is not the zero polynomial and
$\ann M\supsetneq\langle H\rangle$, then $M$ is
finite-dimensional.  If $\cp$ is the zero
polynomial and $\ann M\supsetneq
\langle u\rangle$ or $\ann M\supsetneq 
\langle d\rangle$, then $M$ is one-dimensional.
\end{lem}
\begin{proof}
Suppose first that $\cp=0$ and
$\ann M\supsetneq \langle d\rangle$.
Then the previous lemma implies that 
$u^k\in\ann M$ for some $k\geq 1$;
this implies that $u\in\ann M$
(since either $uM=0$ or $uM=M$). Hence
$\ann M$ contains both $u$ and $d$, and
therefore all three operators $u$, $d$, and $h$
commute, so they all act as scalars. Thus
$M$ is one-dimensional. Similarly,
if $\cp=0$ and
$\ann M\supsetneq \langle u\rangle$,
then $M$ is one-dimensional. 

On the other hand, suppose now
that  $\cp\neq 0$. By the previous lemma
$\ann M$ contains $u^k$ or $d^k$;
we'll say $u^k$ for definiteness. 
We first show that
$\cp$ cannot be a constant. If $\cp$
were the constant $C$, then $ud=H-C$
would act as $-C$ on $M$, and
thus $u^kd^k$ would act as the
nonzero constant $(-C)^k$ on $M$,
contradicting $u^kM=0$. Thus
$\cp$ must be a nonconstant polynomial.
Let $u^m$ be the smallest
power of $u$ that lies in $\ann M$. In this
case,
\[
du^m-s^mu^md=[s^{m}\cp(h-(m-1)\gamma)-\cp(h+\gamma)]u^{m-1}.
\]
The polynomial $s^{m}\cp(h-(m-1)\gamma)-\cp(h+\gamma)$
is nonzero, since $s^{m}\cp(h-(m-1)\gamma)=\cp(h+\gamma)$
implies that $\cp$ has infinitely many roots. (If $\alpha$
is a root of $\cp$, then so are $\alpha+m\gamma$,
$\alpha+2m\gamma$, and so on.)
Thus we can use Corollary~\ref{findimcor} to
conclude that $M$ is finite-dimensional.  
\end{proof}

\subsection{The case $HM=M$}
When $HM=M$, the situation is more complicated, 
especially when $o(s)<\infty$. If $o(s)=n$, 
then
$H^nu=uH^n$ and $dH^n=H^nd$,
so $H^n$ is a central element in $L$.
Thus $H^n$ acts a  scalar $c^n$
for some constant $c$. (We
write the constant as an $n$th power
for balance with $H^n$.) Thus any
primitive ideal must contain 
$H^n-c^n$; since $HM=M$,
we have $c\neq 0$.  There is
a special case, however, that we need to look at
more closely. It is complicated enough
that we bestow upon it a separate lemma.
\begin{lem}\label{specialcase}
Suppose $o(s)=n$, $\cp$
is a constant polynomial $C\neq 0$,
and $H^n$ acts as the scalar $C^n$ on
a simple module $M$.
Then $\ann M$ must contain
either $u^n$ or $d^n$.
\end{lem}
\begin{proof}
We have $\phi(h)=(s-1)C$, so 
$du=sud+(s-1)C$. 
By induction we get that 
$d^iu=s^iud^i+(s^i-1)C$ for $i\geq 1$.
In particular, $d^nu=ud^n$. 
Thus $ud^nM=d^nuM$;
also, $hd^nM=d^n(h-n\gamma)M$.
Therefore $d^nM$ is submodule of $M$,
hence either $d^nM=0$
or $d^nM=M$. Similarly, either
$u^nM=0$ or $u^nM=M$. 
(Note that this result holds for
any simple module $M$, not
just those where $H^n$ acts
as $C^n$.)

Recall that $ud=H-C$. By
induction we establish that 
$u^kd^k=\prod_{i=0}^{k-1} (s^{-i}H-C)
=(-1)^k\prod_{i=0}^{k-1} (C-s^{-i}H)$. Thus
$u^nd^n=(-1)^n\prod_{i=0}^{n-1}(C-s^{-i}H)
=(-1)^n (C^n-H^n)$.
Since $H^n$ acts as $C^n$
on $M$, we conclude that $u^nd^n$ acts
as the zero operator on $M$.
Therefore either $u^nM=0$ or $d^nM=0$,
i.e., $\ann M$ contains either
$u^n$ or $d^n$.
\end{proof}
We remark here that in the situation
of the lemma, any ideal that
contains $u^n$ or $d^n$ must also
contain $H^n-C^n$, since $H^n-C^n=
(-1)^{n+1} u^nd^n$.

Summarizing so far, we have the following information
when $HM=M$: if $o(s)=n$, then
$\ann M$ must contain $H^n-c^n$ for some
$c\neq 0$; furthermore, if $\cp=C\neq 0$
and $c^n=C^n$, then $\ann M$ actually
contains $u^n$ or $d^n$. In other cases
(i.e., if $s$ is not a root of unity), we have
no information about $\ann M$. 

We now look at primitive ideals that
are possibly larger than the minimal ones.
We start with a result about polynomials
of two variables: it 
is the two-variable
analogue of the fact that if $f(x+\gamma)
=f(x)$, then $f$ is a constant polynomial.
(We used this fact in the proof of Lemma~\ref{bigideals2}.)
\begin{lem}\label{polyproperty}
Suppose $g(x,y)=\sum_{i=0}^m g_i(x)y^i$ with
$m<o(s)$, and  $g(x+\gamma,sy)= s^n g(x,y)$
for some $0\leq n\leq m$. 
Then $g(x,y)=cy^n$ for some constant
$c\in\C$.
\end{lem}
\begin{proof}
We have
$\sum_{i=0}^m g_i(x+\gamma) s^i y^i=
\sum_{i=0}^m s^n g_i(x) y^i$.
Comparing the coefficients of $y^i$, we see
that $s^ig_i(x+\gamma) = s^n g_i(x)$. Now
if $g_i$ has a root $a$, then $a+\gamma$,
$a+2\gamma$, \dots, are also roots. Thus $g_i$
would have infinitely many roots, and hence must
be the zero polynomial. If $g_i$ does not have any
roots, then it is a constant $c$, but then we would
have $s^i c = s^n c$, so $c=0$ unless $i=n$.
We conclude that $g_i(x)=0$ if $i\neq n$, but
$g_n(x)=c$. Therefore $g(x,y)=cy^n$, as
required. 
\end{proof}
\begin{lem}\label{bigideals}
Let $M$ be a simple module with $HM=M$.
If $s$ is not a root of unity, and $\ann M\neq \{0\}$,
then $\ann M$ contains a power of $u$
or a power of $d$. If $o(s)=n$,
and $\ann M$ strictly contains $\langle
H^n-c^n\rangle$ for $c\neq 0$, then 
$\ann M$ also contains a power of $u$
or a power of $d$.

Additionally, suppose $\cp(h)=C\neq 0$
and $\ann M\supsetneq \langle d^n
\rangle$. Then $u^n\in\ann M$. Similarly,
if $\ann M\supsetneq \langle u^n
\rangle$, then $d^n\in\ann M$.
\end{lem}
\begin{proof}
If $o(s)=\infty$, let $I=\{0\}$; if $o(s)=n$,
let $I=\langle H^n-c^n\rangle$, where 
$0\neq c\in\C$. Suppose that
$\ann M\supsetneq I$.
We proceed as in the proof of 
Lemma~\ref{bigideals2}. 
Choose an element $x\in\ann M$
that is not in $I$; by Lemma~\ref{homogen}
we can assume that $x$ is homogeneous,
say of degree $k$. For definiteness
we take $k\geq 0$; the case $k<0$
is similar. Thus $\ann M$ contains
an element of the form 
$x=u^kg(h,H)$. If $o(s)=n$,
i.e., $I=\langle H^n-c^n\rangle$,
we can assume that the highest
power of $H$ that appears in $g$
is at most $n-1$.  Among
all nonzero elements of this form,
choose one where the degree of $g$ is minimal.
Write $g(h,H)=\sum_{i=0}^m f_i(h)H^i$;
here $m<o(s)$.
Then $x u=u^kg(h,H)u=u^{k+1}g(h+\gamma,sH)$,
so
\[
x u-s^mux = u^{k+1}[g(h+\gamma,sH)-s^mg(h,H)]
=u^{k+1}\sum_{i=0}^m [s^if_i(h+\gamma)-s^mf_i(h)]H^i.
\]
The polynomial 
$
g'(h,H)=\sum_{i=0}^m [s^if_i(h+\gamma)-s^mf_i(h)]H^i
$
has smaller degree than $g(h,H)$ (if
$g(h,H)$ has degree $(\ell,m)$, then $g'(h,H)$
 has degree at most 
$(\ell-1,m)$), so by choice of $g(h,H)$,
we must have $g'(h,H)=0$. Hence
$g(h+\gamma,sH)=s^mg(h,H)$, and by
Lemma~\ref{polyproperty} we conclude
that $g(h,H)$ is a nonzero multiple
of $H^m$. 
Thus $\ann M$ contains $u^kH^m$. Since
$HM=M$, we conclude that $\ann M$
contains $u^k$. 

Now suppose $o(s)=n$, $\cp(h)=C\neq 0$, $H^n$
acts as the scalar $C^n$ on $M$, and 
$\ann M\supsetneq \langle d^n\rangle$. 
Note that $\langle d^n\rangle$
contains $d^{n-j}d^ju^j=
d^{n-j}g_j(H)$, where $g_j(H)
=\prod_{i=1}^{j} (s^i H-C)$
is a polynomial of degree $j$.

As before, we use Lemma~\ref{homogen}
to conclude that $\ann M$ contains
a homogenous element $x\not\in
\langle d^n\rangle$ of degree $k$. Clearly $k>-n$.
If $k\geq 0$, then we proceed as above,
concluding that $\ann M$ contains
a power of $u$. 
Suppose however that $k<0$. Set
$k=j-n$. Then we can
write $x=d^{n-j}f(h,H)$; since
$d^{n-j}g_j(H)\in I$, we can assume
that the degree of $H$ in $f(h,H)$
is at most $j-1$. Therefore the element
$g(h,H)=u^{n-j}x=u^{n-j}d^{n-j}f(h,H)$
is a nonzero element in $\ann M$;
the degree of $H$ in this element is
at most $n-1$. We can write
$g(h,H)=\sum_{i=0}^m f_i(h)H^i$
where $m<n$. We now proceed as
above, concluding that $\ann M$
contains a power of $u$. Since either
$u^nM=0$ or $u^nM=M$, we
conclude that $u^n\in\ann M$.

Similarly, if $\ann M\supsetneq 
\langle u^n\rangle$, then $\ann M$
also contains $d^n$.
\end{proof}
\begin{lem}\label{findimresult}
Let $M$ be a simple module with $HM=M$.

If $o(s)=\infty$ 
and $\ann M\neq \{0\}$, then $M$ is
finite-dimensional.

If $o(s)=n$, $\cp$ is not a constant polynomial,
and $\ann M\supsetneq \langle H^n-c^n\rangle$
($c\in\C^\times$), then $M$ is finite-dimensional.

If $o(s)=n$, $\cp(h)=C$ is constant,
and $\ann M\supsetneq \langle H^n-c^n\rangle$
($c\in\C^\times$), then $c^n=C^n$. (Thus
$\ann M$ must contain $u^n$ or $d^n$.)

If $o(s)=n$, $\cp(h)=C$ is constant,
and $\ann M\supsetneq \langle u^n\rangle$
or $\ann M\supsetneq \langle d^n\rangle$,
then $M$ is finite-dimensional. 
\end{lem}
\begin{proof}
In all the cases listed, the preceding
lemma implies that 
$\ann M$ contains a power of $u$
or a power of $d$. For 
definiteness let's say that $u^k\in \ann M$.

Suppose now that $o(s)=\infty$ and $\ann M
\neq \{0\}$.  We first show that $\cp$ cannot
be a constant polynomial. 

If $\cp$ is the zero polynomial, then $H=ud$
and $u^kd^k=\prod_{i=0}^{k-1} s^{-i}H$ is in
$\ann M$, i.e., $H^k$ is in $\ann M$. This 
contradicts $HM=M$.

If $\cp$ is a nonzero constant $C$, we have
to work harder. Let $v_0$ be a nonzero element
in $\ker u$ (such an element exists since $u^kM=0$).
Let $v_i=d^iv_0$. We check easily that for $i\geq 1$,
$uv_i=(s^{-i}-1)Cv_{i-1}$. In particular, 
$u^kv_k=cv_0$, where $c=C^k\prod_{i=1}^k (s^{-i}-1)\neq 0$.
This contradicts $u^k\in\ann M$.

Thus $\cp$ is nonconstant. We now use 
Corrolary~\ref{findimcor}
to conclude that 
$M$ is finite-dimensional.

Similarly, if $o(s)=n$, $\cp$ is nonconstant
and $\ann M\supsetneq \langle H^n-c^n\rangle$,
we use Corrolary~\ref{findimcor}
to conclude that 
$M$ is finite-dimensional.

So it remains to consider the case where
$o(s)=n$ and 
$\cp(h)=C$ is a constant polynomial. 
Now  $u^kM=0$
implies that $u^n\in\ann M$, so $u^nd^n=\pm(H^n-C^n)$
is also in $\ann M$. Thus the scalar $c^n-C^n$
is in $\ann M$, hence $c^n=C^n$, as required.

Suppose now that $\cp(h)=C$,
$\ann M$ strictly contains $\langle u^n\rangle$
or $\langle d^n\rangle$. By the lemma above,
$\ann M$ contains both $u^n$ and $d^n$.
We now show that if $u^n, d^n\in \ann M$,
then $M$ is finite-dimensional. First,
choose a nonzero element $v_0$ such
that $dv_0=0$ (such an element
clearly exists). Define $v_i=u^iv_0$;
then $v_j=0$ for $j\geq n$ since
$u^n\in\ann M$. Note that
$Hv_0=(ud+C)v_0=Cv_0$. It follows
that $Hv_i=Hu^iv_0=s^iu^iHv_0=s^iCv_i$
for $0\leq i\leq n-1$. 
Since $h$ commutes with
$H$, we have $Hf_i(h)v_i=s^if_i(h)v_i$
for any polynomial $f_i(h)$. 
Thus each $f_i(h)v_i$
belong to a different eigenspace
of $H$, $0\leq i\leq n-1$.

Now $M$ is simple, so there exists
an element $x\in L$ such that $xhv_0=v_0$.
We can rewrite this equation as
$h\sum_{i=0}^{n-1} f_i(h)u^i v_0=v_0$
for some polynomials $f_i$. Thus
$h\sum_{i=0}^{n-1}f_i(h)v_i=v_0$.
For $i\neq 0$, the $H$-eigenvalue
of  $f_i(h)v_i$ is
different from the $H$-eigenvalue
of $v_0$, so we must have $f_i(h)v_i=0$.
Thus $hf_0(h)v_0=v_0$, or
$(hf_0(h)-1)v_0=0$. In other words,
we have found a nontrivial polynomial
in $h$ that annihilates $v_0$.
By Corollary~\ref{findimcor}, this implies
that $M$ is finite-dimensional.
\end{proof}
\subsection{Proving primitivity}
We now show that the ideals
$\{0\}$, $\langle H\rangle$, $\langle H^n-c^n\rangle$,
$\langle u\rangle$, $\langle d\rangle$,
$\langle u^n\rangle$, and $\langle d^n
\rangle$ are all primitive, under the
right circumstances.
\begin{lem}
The following are primitive ideals:
\begin{enumerate}
\item $\langle H\rangle$, if $\cp$ is not identically zero;
\item $\{0\}$, if $o(s)=\infty$;
\item $\langle H^n-c^n\rangle$ ($c\neq 0$), if $o(s)=n$
and $\cp$ is not a constant polynomial;
\item $\langle H^n-c^n\rangle$ ($c\neq 0$), if $o(s)=n$
and $\cp$ is a constant $C\neq c$;
\item $\langle u\rangle$ and $\langle d\rangle$,
if $\cp(h)=0$;
\item $\langle u^n\rangle$ and $\langle d^n\rangle$,
if $o(s)=n$ and $\cp$ is a nonzero constant $C$.
\end{enumerate}
\end{lem}
\begin{proof}
We make use of the universal weight module
$W(\lambda,\beta)$
from section~\ref{universal}. In the present
case ($r=1,\gamma\neq 0$), we have
$\lambda_i=\lambda+i\gamma$
and $\beta_i=s^i\beta+s^i\cp(\lambda)
-\cp(\lambda+i\gamma)$. Note that
the $\lambda_i$s are always all distinct, so 
if $\beta_i\neq 0$
for all $i\in\Z$, then 
$W(\lambda,\beta)$ is simple. 

We need to know how $H=ud+\cp(h)$
acts on $W(\lambda,\beta)$. A quick
calculation shows that 
for $i\in\Z$, $Hv_i=s^ic$, where
$c=\beta+\cp(\lambda)$. We then
have $\beta_i=s^ic-\cp(\lambda+i\gamma)$.
We usually set a value for $c$ first,
then find $\lambda$ and $\beta$
to suit our purpose.

Suppose first that $\cp$ is a nonconstant
polynomial. Then no matter what
the value of $c$ is, we can find
$\lambda$ and $\beta$ such 
that $\beta+\cp(\lambda)=c$
and $\beta_i\neq 0$. 
(The details:
For a given $i\in\Z$, the polynomial equation
$\cp(x)=s^ic$ has only finitely many
solutions, so there are only
countably many values of $\lambda$
where $\cp(\lambda+i\gamma)=s^ic$
for some $i\in\Z$. Thus we can
certainly choose a value of $\lambda$
so that $\cp(\lambda+i\gamma)\neq s^ic$
for \emph{any} $i\in\Z$. We can then
choose $\beta$ to satisfy $\beta+\cp(\lambda)
=c$.) Thus if $\cp$ is a nonconstant
polynomial, we can construct a simple
$W(\lambda,\beta)$ for any value of $c$.

When $c=0$, we get a simple module
$M$ with $HM=0$. Additionally, note that even if
$\cp$ is a constant $C\neq 0$, we can still
choose $\lambda$ and $\beta$ so that
$c=\beta+C=0$ and $\beta_i=-C\neq 0$
for all $i\in\Z$. Thus we get a simple
module $M$ with $HM=0$ whenever
$\cp$ is nonzero. Its annihilator is
$\langle H\rangle$ by Lemma~\ref{bigideals2}.
This proves (1). 

When $c\neq 0$ and $o(s)=n$,
we get a simple module $M$ with $HM=M$;
its annihilator must be $\langle H^n-c^n\rangle$
by Lemma~\ref{bigideals}. This proves (3). 

When $c\neq 0$ and $o(s)=\infty$, again
we get a simple module with $HM=M$.
Note that even if $\cp(h)=C$, a constant
(which could be zero),
we can still choose $\lambda$ and $\beta$
so that $c=\beta+C\neq 0$ and 
$\beta_i=s^ic-C\neq 0$
for all $i\in\Z$. Thus we have a simple
module $M$ with $HM=M$ no matter
what $\cp$ is. The annihilator of 
this module is $\{0\}$, by Lemma~\ref{bigideals}.
This proves (2). 

Suppose now that $c\neq 0$, $o(s)=n$,
and $\cp$
is a constant $C$. We can
still choose $\lambda$ and $\beta$
such that
$\beta+C=c$ and
$\beta_i=s^ic-C\neq 0$
for any $i\in\Z$, provided $C$ is not
one of the values $s^ic$, i.e., provided
$C^n\neq c^n$. We get a simple
module $M$ such that $HM=M$.
The annihilators must be
$\langle H^n-c^n\rangle$  by 
Lemma~\ref{bigideals}. This proves (4). 

It remains to prove (5) and (6). 
We now need to introduce modules that
are not weight modules.
Here is one possibility. Suppose
$\cp(h)=C$ where $C\in\C$.
Pick a positive integer $n$. 
Let $M$ be a module with basis
$\{v_i: i\in\Z\}$ where
\[
uv_i=v_{i+1}, \quad hv_i=v_{i-n}+i\gamma v_i,\quad
dv_i=(s^i-1)Cv_{i-1}. 
\]
We can check that
this is indeed a module:
\begin{gather*}
huv_i=v_{i+1-n}+(i+1)\gamma v_{i+1}=
v_{i-n+1}+i\gamma v_{i+1} + \gamma v_{i+1}=
u(h+\gamma)v_i;\\
dhv_i=(s^{i-n}-1)Cv_{i-n-1}+i\gamma (s^i-1)Cv_{i-1}
=(h+\gamma)dv_i;\\
[sud+(s-1)C]v_i=[s^{i+1}-s]Cv_i + [s-1]Cv_i = [s^{i+1}-1]C=
duv_i.
\end{gather*}
We now verify that this module is simple.
Define the operator
$T=uh-1$. Its action on $v_i$ is $Tv_i = i\gamma v_{i+n}$.
Thus $v_i$ is ``almost'' an eigenvector of $T$ with
eigenvalue $i\gamma$. In this case ``almost''
is good enough. 

Suppose $M_1$ is a nonzero submodule of $M$. 
Let $x=\sum a_iv_i$ be a nonzero element in
$M_1$ with minimal length. 
Then $Tx = \sum i\gamma a_i v_{i+n}$ and
for any $m$ with $a_m\neq 0$, we have
$Tx-m\gamma ux = \sum (i-m)\gamma
 a_i v_{i+n}$, which has shorter length than $x$, and
 so must be zero. Thus $x=a_mv_m$, i.e.,
 $M_1$ contains a pure basis vector $v_m$.
 
 Now $M$ can be generated by just one $v_m$.
 To get $v_{m+1}$ we simply apply $u$,
 and to get $v_{m-1}$ we apply $h$ and subtract
 $m\gamma v_m$, getting $v_{m-n}$, then
apply $u$ repeatedly until we get $v_{m-1}$.
 Thus $M_1=M$. So $M$ is indeed
 a simple module.

Now if $C=0$ and $n=1$, then $d$ annihilates
$M$. By Lemma~\ref{bigideals2}, $\ann M$
cannot be bigger than $\langle d\rangle$.
Thus $\langle d\rangle$ is indeed a primitive
ideal. Similarly, $\langle u\rangle$ is a
primitive ideal. This proves (5).

If $C\neq 0$ and $o(s)=n$, then 
$d^n$ acts as the zero scalar.
Thus $\ann M$ contains $\langle d^n
\rangle$. It cannot be any larger by
Lemma~\ref{bigideals}. Thus $\langle d^n
\rangle$ is a primitive ideal.
Similarly, $\langle u^n\rangle$ is a primitive
ideal. This proves (6). 
\end{proof}

\subsection{A list of primitive ideals}
Putting all these together, we get a complete
list of the primitive ideals of $L(\phi,1,s,\gamma)$
($\gamma\neq 0$),
summarized in the following table. Note
that we only include primitive ideals that
are not annihilators of finite-dimensional
modules.
\bigskip
\begin{center}
\begin{tabular}{|c|c|c|}\hline
$o(s)$ & $\cp$ & primitive ideals \\ \hline
$\infty$ & $0$ & $\{0\}, \langle u\rangle, \langle d\rangle$\\
$\infty$ & nonzero & $\{0\}, \langle H\rangle$\\
$n$ & $0$ & $\langle H\rangle, \langle u\rangle, \langle d\rangle,
			\langle H^n-c^n\rangle$ ($c\neq 0$)\\
$n$ & $C\neq 0$ & $\langle H\rangle, \langle u^n\rangle, \langle d^n\rangle,
			\langle H^n-c^n\rangle$ ($c\neq 0, c^n\neq C^n$)\\
$n$ & nonconstant & $\langle H\rangle, \langle H^n-c^n\rangle$ ($c\neq 0$)\\ \hline
\end{tabular}
\end{center}
\bigskip
The preceding process of determining primitive ideals
of $L(\phi,1,s,\gamma)$ is fairly typical of our method
in general. Using a central element or otherwise,
we find an element that must be present in
all primitive ideals. Using results similar to
Lemmas~\ref{homogen} and~\ref{containspower}
we hope to show that any primitive ideal
larger than these minimal ideals must contain
a power of $u$ or a power of $d$; usually
this is enough to conclude that these large
primitive ideals must be annihilators of
finite-dimensional simple modules. Sometimes
we will have to consider special cases that
could very well involve modules that are
not weight modules. 

\section{When $L$ is conformal and $\gamma=0$}

In this section we assume that  $\gamma=0$
and $L$ is conformal.
Recall that this means there
exists a polynomial $\cp$ such that $s\cp(x)-\cp(rx)=\phi(x)$.
We will use $\cp$ throughout instead of $\phi$.

It is very useful, especially when roots of unity
are involved, to have a commutation formula between
$u$ and $d^k$, $k\geq 1$; see Corollary~\ref{findimcor}.
With the base
case $du-sud=s\cp(h)-\cp(rh)$, we establish
by induction that
\[
d^ku-s^kud^k = [s^k\cp(h)-\cp(r^kh)]d^{k-1}.
\]
There is of course a corresponding formula
involving $d$ and $u^k$.

As before we define $H$ to be $ud+\cp(h)$; then 
$Hu=suH$ and $dH=sHd$.
If $M$ is a simple $L$-module, then these relations
imply that either $HM=M$
or $HM=0$; similarly, either $hM=M$ or $hM=0$. 
There are thus four cases to consider: when
$hM=HM=0$, when $hM=M$ and $HM=0$,
when $hM=0$ and $HM=M$, and finally,
when $hM=HM=M$.

\subsection{Homogeneous elements
in ideals}
Before we treat these cases individually,
we state the following result. It is
a somewhat more
complicated analogue of Lemma~\ref{homogen}.
\begin{lem}\label{homogen2}
Suppose $o(r)=\infty$
and $x$ is an element
in an ideal $I$ of $L$. Let $x=\sum_{i\in\Z} x_i$ be
its homogeneous decomposition (i.e., $x_i\in L_i$).
Then $x_ih^{\ell(x)-1}\in I$ for all $i\in\Z$, where
$\ell(x)$ is the length of $x$. 
\end{lem}
\begin{proof}
As before, we use induction on $\ell(x)$. The claim is
certainly true if $\ell(x)=1$. Assume that $\ell(x)>1$.
Then $hx=\sum hx_i=\sum r^ix_ih$.
Pick $k\in\Z$ such that $x_k\neq 0$. Then
\[
hx-r^kxh=\sum_{i\in \Z}(r^i-r^k)x_ih.
\]
Since $r^i\neq r^k$ unless $i=k$, we have
$\ell(hx-s^k xh)=\ell(x)-1$. 
By the induction hypothesis, $(r^i-r^k)x_ihh^{\ell(x)-2}\in I$;
hence $x_ih^{\ell(x)-1}\in$ for all $i\neq k$. 
Thus
$y=\sum_{i\neq k} x_ih^{\ell(x)-1}\in I$, so 
$x_kh^{\ell(x)-1}=xh^{\ell(x)-1}-y\in I$ also. 
\end{proof}
\begin{rem}\label{homogen2rem}
The proof only uses the commutation relation
between $h$ and elements of degree $i$.
Since $H$ has similar commutation relations,
the result is also true when $h$ is replaced
by $H$ and $r$ by $s$. 
\end{rem}
We now consider each of the four cases in turn.

\subsection{When $hM=HM=0$}
We dispose of this case immediately.
Both $h$ and $ud=H-\cp(h)$
act as scalars on $M$, so by Lemma~\ref{onedim},
$M$ must be one-dimensional. 

\subsection{When $hM=M$ and $HM=0$}
Suppose first that $o(r)=n$. Then 
$h^n$ commutes with $u$ and $d$,
and hence $h^n$ must act as a scalar
on $M$. This implies that $h$ has
an eigenvector; thus $hv=\lambda v$
for some nonzero $v\in M$ and
$\lambda\in\C$. Note that $Hv=0$,
so $(ud)v=-\cp(\lambda)v$, i.e.,
$v$ is also an eigenvector of $ud$.
Thus $v$ is a weight vector. 
Now recall that
\begin{align*}
d^ku&=s^kud^k + [s^k\cp(h)-\cp(r^kh)]d^{k-1}\\
&=s^k[H-\cp(h)]d^{k-1}+[s^k\cp(h)-\cp(r^kh)]d^{k-1}\\
&=[s^kH-\cp(r^kh)]d^{k-1}
\end{align*}
for $k\geq 1$. In particular, since $H$
acts as zero on $M$, we have $d^nu=
-\cp(h)d^{n-1}$ as operators on $M$.
Compare this with $ud=H-\cp(h)$,
i.e., $ud^n=Hd^{n-1}-\cp(h)d^{n-1}$,
and we conclude that as operators on
$M$, $d^n$ commutes with $u$. Clearly
$d^n$ also commutes with $h$, and 
hence $d^n$ acts as a scalar on $M$.
(Similarly, $u^n$ acts as a scalar on $M$.)
By Lemma~\ref{findim}, $M$ is finite-dimensional.

Thus we can concentrate
on the case where $r$ is not a root of unity.
Clearly $\ann M$ contains
$\langle H\rangle$, but when $\cp(h)=0$,
we can say even more. 
As in section~\ref{HM=0}, in this
case $\ann M$ must actually
contain $u$ or $d$. 
\begin{lem}\label{conf:containspower}
Suppose $r$ is not a root of unity, 
$HM=0$, $hM=M$, and $\ann M\supsetneq 
\langle H\rangle$. Then $\ann M$ contains
a power of $u$ or a power of $d$. 

When in addition $\cp(h)=0$ and $\ann M
\supsetneq \langle d\rangle$, then $\ann M$
contains $u$. Similarly, if $\ann M
\supsetneq \langle u\rangle$, then $\ann M$
contains $d$. 
\end{lem}
\begin{proof}
Denote by $J$ the ideal $\langle H\rangle$;
in the special case where $\cp(h)=0$, let
$J$ denote the ideal $\langle d\rangle$. 
Pick an element $x$ in $\ann M$ that is
not in $J$. 
By Lemma~\ref{homogen2}, 
$\ann M$ contains an element of the form
$x_ih^k$ where $x_i\not\in J$,
$x_i\in L_i$. Since $hM=M$, we conclude
that $x_i\in\ann M$ also. In the special
case where $\cp(h)=0$, we clearly
have $i\geq 0$. In the other cases, to be definite
we also assume that $i\geq 0$; the case
where $i\leq 0$ is similar. Thus $\ann M$
contains an element of the form $u^if(h)$
for some nonzero polynomial $f$. Among all
such elements, choose one where the degree
of $f$ is as small as possible, say $m$. Then
\[
u^if(h)u - r^m u^{i+1}f(h)=u^{i+1}[f(rh)-r^mf(h)]
\]
is also in $\ann M$. But the polynomial
$f(rh)-r^mf(h)$ has smaller degree than $m$,
hence we must have $f(rh)=r^mf(h)$. If
$f(h)=\sum_{j=0}^m a_ih^i$, then
this implies that $\sum_{j=0}^m (r^j-r^m)a_jh^j=0$.
Since $r$ is not a root of unity, we have
$a_j=0$ for $j\neq m$.  Thus $f(h)$ is 
a multiple of $h^m$. Since $hM=M$, we conclude that
$\ann M$ contains $u^i$. 

In the case where $\cp(h)=0$, recall that
either $uM=M$ or $uM=0$; since $u^i\in\ann M$,
we must have $uM=0$, i.e., $u\in \ann M$. 

The other cases of the lemma are treated
similarly.
\end{proof}
The lemma has immediate consequences
for large primitive ideals.
\begin{lem}\label{confbigideals}
Suppose $o(r)=\infty$, $M$ is a simple module, $HM=0$, 
and $hM=M$. If $\ann M\supsetneq \langle H\rangle$
and $\cp$ is not zero, then $M$ is finite-dimensional.
If $\cp$ is zero and $\ann M\supsetneq
\langle u\rangle$ or $\ann M\supsetneq
\langle d\rangle$, then $M$ is one-dimensional. 
\end{lem}
\begin{proof}
By the previous lemma, $\ann M$ contains
a power of $u$ or a power of $d$; assume it
is a power of $d$ for definiteness. Then
there is a nonzero element $v\in M$ such
that $dv=0$. Since $H=ud+\cp(h)$ and
$HM=0$, it is clear that $\cp(h)v=0$ also.
By Corollary~\ref{findimcor}, $M$ is finite-dimensional.

If $\cp$ is zero, then the previous lemma
implies that $\ann M$ contains both 
$u$ and $d$. Thus as operators on $M$,
the elements $u$, $d$, and $h$ all commute,
so $M$ must be one-dimensional.
\end{proof}

\subsection{When $hM=0$ 
and $HM=M$}
This case is similar to
the previous one, with $H$ and $s$
interchanged with $h$ and $r$. The
conclusions, however,
are different, 
so we will go through the details.

As before, we start by looking at the
case when $o(s)=n$. The operator
$\cp(h)$ acts as the scalar $\cp(0)$
on $M$, so in this case we have
 $d^ku=s^kud^k +(s^k-1)\cp(0)d^{k-1}$.
Thus
$d^nu=ud^n$, i.e., $d^n$ commutes with
$u$ (and with $h$) as an operator on $M$. 
Thus $d^n$ acts as a scalar on $M$.
Similarly, $u^n$ acts as a scalar on $M$. 

Note that $H^n$ also commutes with
$u$, $d$, and $h$, and hence acts as
a scalar on $M$. Thus $H$ has an eigenvector
on $M$. Since $h$ acts as a scalar on $M$
(namely, as zero), we can use Lemma~\ref{findim}
to conclude that $M$ is finite-dimensional. 

We can thus concentrate on the case
where $s$ is not a root of unity. Clearly
$\ann M$ contains $\langle h\rangle$.
It turns out that $\ann M$ cannot be
larger than $\langle h\rangle$ in this case. 
\begin{lem}\label{notlarger}
Suppose $s$ is not a root of unity, and $M$
is a simple module with $HM=M$,
$hM=0$. Then $\ann M=\langle h\rangle$.
\end{lem}
\begin{proof}
Assume that $\ann M\supsetneq \langle
h\rangle$. We will derive a contradiction.
Using the same calculation as in 
Lemma~\ref{conf:containspower}, with
$H$ and $s$ exchanging roles with
$h$ and $r$, we conclude that
$\ann M$ contains a power of $u$
or a power of $d$. For definiteness
we assume that $\ann M$ contains
$d^m$ but not $d^{m-1}$,
where $m\geq 1$. Then $\ann M$
also contains
\[
d^mu-s^mud^m = (s^m-1)\cp(0)d^{m-1}.
\]
If $\cp(0)\neq 0$ we reach our contradiction.
If $\cp(0)=0$, then $ud=H$ as operators
on $M$. Thus $u^md^m=s^{-m}H^m$,
but this is also a contradiction since the
$u^md^m$ acts as zero while $H^m$ 
cannot act as zero since $HM=M$. 
\end{proof}
\subsection{When $hM=M$ and $HM=M$}
As before, we first consider the case where
both $r$ and $s$ are roots of unity. In this case
there exists a positive integer $n$ such that 
$r^n=s^n=1$. Then $h^n$ and $H^n$ are
both central elements, so they both act
as scalars on $M$. Since $h$ and $H$
commute, they have a common eigenvector.
Therefore $h$ and $ud=H-\cp(h)$ also
has a common eigenvector.

Recall that $d^ku-s^kud^k = [s^k\cp(h)-\cp(r^kh)]d^{k-1}$,
which implies that $d^nu=d^nu$.
Clearly $d^nh=r^nhd^n=hd^n$,
thus $d^n$ is a central element.
Hence $d^n$ also acts as a scalar
on $M$. Similarly, $u^n$ also acts
as a scalar on $M$. We can apply
Lemma~\ref{findim}
and conclude that $M$ is actually
finite-dimensional. 

We thus can concentrate on the case
where at least one of $r$ and $s$
is not a root of unity. In this case we
can use Lemma~\ref{homogen2}
and the remark following it, either
through $H$ and $s$, or through
$h$ and $r$.

Since now both $h$ and $H$ act
nontrivially on $M$, it is possible that
they interact in somewhat complicated
ways. To help keep track of this
interaction we make use of
the set  $\Se(r,s)=\{(i,j)\in\Z\times \Z:
r^i=s^j\}$. 
We will often just write $\Se$ for
$\Se(r,s)$, since $r$ and $s$
are usually fixed.

Note that $\Se$ is an additive
subgroup of $\Z\times \Z$. Since
we are assuming that not both
of $r$ and $s$ are roots of unity,
it is straightforward to verify that
$\Se$ is in fact generated by
one element: there exits
$(n,m)\in\Se$ such that 
$\Se=\{(kn,km): k\in\Z\}$. We can
thus divide the sets $\Se$ into
two types: those that contain pairs
$(i,j)$ where $i$ and $j$ have
the same sign or zero,
and those that contain pairs $(i,j)$
where $i$ and $j$ have
opposite signs. For ease
of reference, we adopt the
following notational convention:
when $\Se$ is of the first type,
we write $\Se=\langle (n,m)\rangle$
to indicate that $\Se$ is
generated by $(n,m)$ where
$n,m\geq 0$; when $\Se$
is of the second type, we write
$\Se=\langle (n,-m)\rangle$
to indicate that $\Se$ is
generated by $(n,-m)$
where $n,m>0$. Thus the
type of $\Se$ is indicated
by the presence or absence
of the minus sign on $m$.

Now let $M$ be a simple
module. Every element of $\Se$
gives rise to an element of $\ann M$,
as follows. 

Suppose first that $\Se =
\langle (n,m)\rangle$ and $(i,j)\in
\Se$. The pair $(-i,-j)$ gives
rise to the same element of 
$\ann M$, so we can assume that $i,j\geq 0$.
Since $hM=M$ (and
$\ker h=0$), the operator $h$ has
an inverse $h^{-1}$ (although
the element $h$ does not have an
inverse in $L$). We have $h^{-1}u
=r^{-1}uh^{-1}$ and $dh^{-1}=r^{-1}h^{-1}d$.
Thus the operator $h^{-i}H^{j}$ commutes with 
everything in $L$, and hence acts as a scalar
$c_{i,j}$ on $M$. The scalar $c_{i,j}$ must be nonzero
because $hM=M$ and $HM=M$.
This implies that $H^j$ acts the same way
on $M$ as does $c_{i,j}h^i$. 
Therefore $\ann M$ must contain
$H^j-c_{i,j}h^i$.

We define $I_\Se$ as the ideal
generated by all $H^j-c_{i,j}h^i$
as $(i,j)$ ranges over the elements
of $\Se$; then $I_\Se\subseteq \ann M$.
Since $\Se$ is generated by the
single element $(n,m)$, it is straightforward
to verify that
if $(i,j)=(kn,km)$, then 
$H^j-(c_{n,m})^k h^j\in \langle
H^m-c_{n,m}h^n\rangle \subseteq I_\Se$.
Since $H^j-c_{i,j}h^i\in I_\Se$,
we conclude that 
 $c_{i,j}=(c_{n,m})^k$ and
 that $I_\Se= \langle
H^m-c_{n,m}h^n\rangle$.

Suppose now that $\Se=\langle
(n,-m)\rangle$ and $(i,-j)\in \Se$.
As before, the pair $(-i,j)$ gives
rise to the same element of $\ann M$,
so we can assume that  $i,j>0$.
This time we conclude that 
$h^iH^j - c_{i,-j}\in \ann M$ for
some nonzero scalar $c_{i,-j}$.
A similar calculation as before
shows that $c_{i,-j}=(c_{n.-m})^k$
if $(i,-j)=(kn,-km)$. If we define
$I_\Se$ as the ideal generated 
by all $h^iH^j - c_{i,-j}$ as
$(i,-j)$ ranges over $\Se$, 
then $I_\Se=\langle
h^nH^m-c_{n,-m}\rangle$. 

Here are some examples. Suppose $r$ and $s$ are
algebraically independent. Then $\Se=\{(0,0)\}$
and $I_\Se=\{0\}$. Suppose $o(s)=m<\infty$
and $o(r)=\infty$. Then $\Se=\{(0,km):k\in\Z\}$
and $I_\Se=\langle H^m- c_{0,m}\rangle$. Suppose
$rs=1$ and $o(r)=o(s)=\infty$. Then 
$\Se=\{(k,-k):k\in\Z\}$ and 
$I_\Se=\langle hH-c_{1,-1}\rangle$. 

Inspired by these observations, we define,
for a given $\Se$ and \emph{any} $c\in\C^\times$,
the ideal $I_c= \langle H^m-ch^n\rangle$ if
$\Se=\langle (n,m)\rangle$, and
$I_c=\langle h^nH^m-c\rangle$ if
$\Se=\langle (n,-m)\rangle$.
Note that if $(i,j)\in\Se$ with $i,j\geq 0$, then 
$(i,j)=(kn,km)$ for some $k\geq 0$,
and the element $H^j-c^k h^i$ lies in $I_c$.
Similarly, if $(i,-j)\in\Se$ with
$i,j>0$, then $(i,-j)=(kn,-km)$ for some $k\geq 0$,
and the element $h^iH^j-c^k$ lies
in $I_c$.

Any primitive ideal of $L$ contains 
$I_c$ for some $c\in\C^\times$.
We need to determine which values of $c$
actually make $I_c$ a primitive ideal.
The easy case, of course, is when
$\Se=\{(0,0)\}$. Then $I_c=\langle
1-c\rangle$ and so only $I_1=\{0\}$
can be a primitive ideal. 

When $\Se\neq \{(0,0)\}$,
the situation is almost the complete
opposite. It turns out
that almost all nonzero values of $c$
make $I_c$ a 
primitive ideal, but as usual there
are special cases where $\ann M$ has to contain
more than just $I_c$.
\begin{lem}\label{specialcase2}
Let $M$ be a simple module with $hM=HM=M$.
Suppose $(jm,m)$ is the generator of $\Se$ for
some $j\geq 0$ and $m>0$ (thus $s^m=r^{jm}$).
Suppose further that  $\cp(h)=Ch^j$ with
$C\neq 0$. If $I_{C^m}\subseteq \ann M$,
i.e., if $H^m-C^mh^{jm}\in \ann M$,
then $\ann M$ also contains $u^m$ or $d^m$.
\end{lem}
\begin{proof}
We note that $s^{-1}r^{j}$ is an $m$th root of unity.
Since $du=sH-Cr^jh^j$, we get by induction that
\[
d^ku^k=\prod_{i=1}^k (s^iH-r^{ij}Ch^j)
=\prod_{i=1}^k s^i[H-(r^js^{-1})^iCh^j].
\]
In particular, for $k=m$, we get that $d^mu^m$ is
a nonzero multiple of $H^m-C^mh^{jm}$.
Thus $d^mu^m$ acts as the zero operator
on $M$. 

Recall that
\[
d^mu-s^mud^m=[s^m\cp(h)-\cp(r^mh)]d^{m-1}.
\]
When $\cp(h)=Ch^j$ and $s^m=r^{jm}$, 
we get 
$d^mu=s^mud^m$. Thus  
$d^mM$ is a submodule of $M$, either $0$
or $M$. Similarly, $u^mM=0$ or $u^mM=M$.
Since $d^mu^m$ acts as the zero operator
on $M$, it cannot be the case 
that both $u^mM=M$ and $d^mM=M$. 
Therefore $\ann M$ contains
either $u^m$ or $d^m$. 
\end{proof}
We now look at what happens when a
primitive ideal is larger than $I_c$, or
in the case of the lemma above,
larger than $\langle u^m\rangle$
or $\langle d^m\rangle$. 

To do so, 
we need the following lemma about two-variable
polynomials, analogous to 
Lemma~\ref{polyproperty}. We first recall the
definition of distinctive polynomials
from~\cite{Praton1}. A set $\mathcal{T}\subseteq
\Z\times\Z$
is called distinctive if $r^is^j\neq r^{i'}s^{j'}$
for all distinct pairs $(i,j), (i',j')\in\mathcal{T}$. 
A polynomial is called \emph{distinctive}
if its powers are in a distinctive set, i.e.,
$g(x,y)=\sum_{(i,j)\in\mathcal{T}} a_{i,j}x^iy^j$
is distinctive if $\mathcal{T}$ is a distinctive
set. (We count the zero polynomial as
distinctive.)

(Distinctive sets are not exotic. Note that
$r^is^j=r^{i'}s^{j'}$ iff $r^{i-i'}=s^{-(j-j')}$, i.e.,
iff $(i-i',-(j-j'))\in\Se$. Thus $(i,j)$ and
$(i',j')$ are in a distinctive set iff 
$(i,-j)$ and $(i',-j')$ are in different
$\Se$-cosets. Therefore $\mathcal{T}$
is a distinctive set iff the elements
$(i,-j)$, where $(i,j)\in\mathcal{T}$,
are different $\Se$-coset representatives.)
\begin{lem}\label{distinctpolyproperty}
Suppose $g(x,y)=\sum_{(i,j)\in\mathcal{T}}
a_{i,j} x^iy^j$ is a distinctive polynomial
and $g(rx,sy)=r^as^b g(x,y)$
where $(a,b)\in\mathcal{T}$. Then $g(x,y)=
cx^ay^b$ for some $c\in\C$.
\end{lem}
\begin{proof}
We have
\[
\sum_{(i,j)\in\mathcal{T}} a_{i,j}(r^is^j-r^as^b)x^iy^j=0.
\]
Since $\mathcal{T}$ is distinctive, we have
$r^is^j\neq r^as^b$ unless $(i,j)=(a,b)$.
The conclusion follows.
\end{proof}
Distinctive polynomials are useful because
any element of $L_0$ is congruent
modulo $I_c$ (where $c\in\C^\times$)
to a distinctive polynomial
in $h$ and $H$. To see
this, suppose $g(h,H)=\sum_{(i,j)\in \mathcal{R}}
a_{i,j}h^iH^j$ is an arbitrary element of $L_0$,
where $\mathcal{R}$ is not necessarily
a distinctive set. We will show that we can
reduce the set $\mathcal{R}$ until we 
get a distinctive set, without
changing the congruence class of $g(h,H)$
mod $I_c$.

So suppose $(i,j), (i',j')\in\mathcal{R}$
with $r^is^j=r^{i'}s^{j'}$. Then 
$r^{i-i'}=s^{j'-j}$, so $(i-i',j'-j)\in\Se$.
Clearly it does no harm to assume that 
$i-i'\geq 0$. If $j'-j\geq 0$, then 
$H^{j'-j}-c^k h^{i-i'}\in I_c$ for some
$k\geq 0$, 
so $H^{j'}h^{i'}\equiv c^k h^{i}H^j$ mod $I_c$.
Thus we can replace the term $h^{i'}H^{j'}$
by $c^kh^iH^j$ without changing the 
congruence class of  $g(h,H)$ mod $I_c$.
 On the other hand,
if $j'-j<0$, then $h^{i-i'}H^{j-j'}-c^k\in I_c$ for
some $k\geq 0$. Thus
$h^{i}H^j\equiv c^k h^{i'}H^{j'}$ mod $I_c$,
and we can still replace the term $h^{i'}H^{j'}$
by $c^k h^iH^j$ without changing the 
congruence class of $g(h,H)$. Clearly we can continue this
process until we obtain a distinctive polynomial.

We use the lemma about distinctive
polynomials to prove the following analogue
of Lemma~\ref{containspower}.
\begin{lem}\label{containspower2}
Let $M$ be a simple module with $hM=HM=M$.
If $\ann M\supsetneq I_c$, then $\ann M$ contains
a power of $u$ or a power of $d$.

In the special case of Lemma~\ref{specialcase2},
suppose $\ann M\supsetneq \langle u^m\rangle$.
Then $\ann M$ contains $d^m$. Similarly,
if $\ann M\supsetneq \langle d^m\rangle$,
then $\ann M$ contains $u^m$. 
\end{lem}
\begin{proof}
By Lemma~\ref{homogen2} $\ann M$ contains
a homogeneous element $x$ not in $I_c$. For definiteness
assume that $x$ has degree $k\geq 0$; the
case $k<0$ is similar. We can write
$x=u^kf(h,H)$ where $f\in\C[x,y]$ is a nonzero
polynomial. Since $x\not\in I_c$, we can assume
that $f$ is a nonzero distinctive polynomial.

Among all elements of the form $u^kf(h,H)$ where
$f$ is nonzero and distinctive, choose one where the degree
of $f$ is minimal. Let's say the degree of $f$ is
$(a,b)$. Then
\[
xu-r^as^b ux = u^kf(h,H)u-r^as^b u^{k+1}f(h,H)
= u^{k+1}[f(rh,sH)-r^as^bf(h,H)].
\]
The polynomial $f(rh,sH)-r^as^bf(h,H)$ is
distinctive and has smaller degree than $f$,
so it must be zero. By Lemma~\ref{distinctpolyproperty},
$f(h,H)$ is a scalar multiple of $h^aH^b$, so $\ann M$
contains $u^kh^aH^b$. Since $hM=HM=M$,
we conclude that $\ann M$ contains $u^k$. 

Suppose now that we are in the special
case of Lemma~\ref{specialcase2}. Assume that
$\ann M\supsetneq \langle u^m\rangle$. As above,
$\ann M$ contains a homogeneous element $x$
of degree $k$ not in $\langle u^m\rangle$; here 
we can assume that $k<m$. We note 
that since $\ann M$ contains $u^m$,
it also contains $d^{m-k}u^m=
d^{m-k}u^{m-k}u^k$. Here $d^{m-k}u^{m-k}
=\prod_{i=1}^{m-k} s^i[H-(r^js^{-1})^iCh^j]$
is a polynomial whose $H$-degree 
is $m-k$. Thus we can write
$x=u^k f(h,H)$ where $f$ is a polynomial
whose $H$-degree is smaller than $m-k$.
Therefore $\ann M$ contains the element
$d^kx=d^ku^kf(h,H)=g(h,H)$, a polynomial
whose $H$-degree is at most $m-1$. Since
$\Se=\{i(jm,m): i\in\Z\}$, the polynomial
$g$ is distinctive. Thus $\ann M$ contains
an element of the form $d^ig(h,H)$
where $g$ is nonzero and distinctive.
By the calculation above, $\ann M$
contains a power of $d$. Since $d^mM=0$
or $d^mM=M$, we conclude that $d^m
\in\ann M$, as required.
\end{proof}
In most cases, if $\ann M$ contains a power
of $u$ or a power of $d$, then $M$ must 
be finite-dimensional. 
\begin{lem}\label{bigidealimplyfindim}
Let $M$ be a simple module with
$hM=HM=M$. Suppose $u^m\in\ann M$
but $u^{m-1}\not\in\ann M$, $m\geq 1$. Then 
$\cp$ is not the zero polynomial and $M$ is
finite-dimensional, except in the special
case of Lemma~\ref{specialcase2}.
The same conclusion holds when $d^m\in\ann M$
and $d^{m-1}\not\in\ann M$. 

In the special case of Lemma~\ref{specialcase2},
if $\ann M$ contains both $u^m$ and $d^m$,
then $M$ is finite-dimensional. 
\end{lem} 
\begin{proof}
If $\cp$ is the zero polynomial, then $du=sud$.
Thus $uM$ is a submodule and hence must be
either $0$ or $M$. (Similarly, either $dM=0$
or $dM=M$.) Since $u^mM=0$, we must have
$uM=0$. Thus $ud=H$ acts as the zero operator
on $M$, contradicting $HM=M$. 

So assume now that $\cp$ is nonzero.
Recall that
\[
d^mu-s^mud^m=[s^m\cp(h)-\cp(r^mh)]d^{m-1}.
\]
If $\cp(h)=\sum_i a_ih^i$, then 
$s^m\cp(h)-\cp(r^mh)=\sum_i a_i(s^m-r^{im})h^i$,
which is nonzero unless $s^m=r^{jm}$ for some
$j\geq 0$ and $a_i=0$ for $i\neq j$, i.e.,
unless we are in the special situation of
Lemma~\ref{specialcase2}.  (Note that
the equality $s^m=r^{jm}$ means that
$r$ cannot be a root of unity, since that would
force $s$ to be a root of unity also, and we
have excluded this possibility. Thus $s^m=r^{jm}$
is true only for at most one value of $j$.) 
We obtain the conclusion from Corollary~\ref{findimcor}. 

Suppose now we are in the situation of
Lemma~\ref{specialcase2}. The proof
in this case is quite similar to the proof
of the last case of Lemma~\ref{findimresult}.
Choose
an element $v\in M$ such that $dv=0$
(such an element exists since $d^m\in
\ann M$). Define $v_i=u^iv$ for $i\geq 0$;
since $u^m\in\ann M$, we have
$u_{m+i}=0$ for $i\geq 0$. 

Now $0=udv_0=(H-Ch^j)v_0$, so
$Hv_0=Ch^jv_0$. Recall that as an
operator on $M$, the operator $h$
has an inverse $h^{-1}$ (although
the element $h$ is not invertible in $L$);
thus $Hh^{-j}v_0=Cv_0$. Since
$Hh^{-j}u=sr^{-j}uHh^{-j}$, we see
that $Hh^{-j}v_i=(sr^{-j})^i v_i$
for $0\leq i\leq m-1$. Each $v_i$
is an eigenvector of the operator
$Hh^{-j}$, with distinct eigenvalues.

Since $M$ is simple, there exists $x\in L$
such that $xhv_0=v_0$. Since $dhv_0=0$,
we can assume that $x$ can be
written as $\sum_i p_i(h)u^i$, so
\[
\sum_i p_i(h)u^ihv_0=
\sum_i r^{-i}p_i(h)hu^iv_0=
\sum_i r^{-i}p_i(h)h v_i = v_0.
\]
For each $i$,  $r^{-i}p_i(h)hv_i$ is an eigenvector
of $Hh^{-j}$ with eigenvalue $(sr^{-j})^i$.
They form a linearly independent set.
Thus we have $p_i(h)=0$ for $i>0$.
Therefore $(p_0(h)h-1)v_0=0$,
i.e., there exists a nontrivial polynomial
in $h$ that annihilates $v_0$. Thus
we can apply Lemma~\ref{findim}
to conclude that $M$ is finite-dimensional.
\end{proof}
\subsection{Proving primitivity} It remains to show
that the ideals $I_c$, $\langle h\rangle$, $\langle H \rangle$,
$\langle u \rangle$, $\langle d\rangle$, 
$\langle u^m \rangle$, $\langle  d^m \rangle$
are all primitive ideals under the right conditions.
\begin{lem}
Suppose $L$ is conformal, with $\gamma=0$.
Then the following ideals are primitive:
\begin{enumerate}
\item $\langle h\rangle$ if $o(s)=\infty$;

\item $\langle H\rangle$ if $o(r)=\infty$ and $\cp$
is not identically zero;

\item $\langle u\rangle$ and 
$\langle d\rangle$ if $o(r)=\infty$ and
$\cp$ is identically zero;

\item $\{0\}$ if $\Se(r,s)=\{(0,0)\}$;

\item $\langle h^nH^m-c\rangle$ if
$c\in\C^\times$ and 
$\Se(r,s)=\langle (n,-m)\rangle$
(where $n, m>0$);

\item $\langle H^m-ch^n\rangle$,
where $c\in\C^\times$, $\Se(r,s)=
\langle (n,m)\rangle$ (where
one of $n$ or $m$ is nonzero),
and we are not in the case of 
Lemma~\ref{specialcase2};

\item $\langle u^m\rangle$ and 
$\langle d^m\rangle$ if
we are in the case of 
Lemma~\ref{specialcase2}, i.e.,
if $\cp(h)=Ch^j$ for some
$C\neq 0$ and $j\geq 0$,  and 
$\Se(r,s)=\langle (jm,m)\rangle$ 
for some $m>0$.
\end{enumerate}
\end{lem}
\begin{proof}
We again make use of the universal weight
module $W(\lambda,\beta)$ described in
section~\ref{universal}. In this case,
it is more convenient to use the parameters
$\lambda$ and $\mu=\beta+\cp(\lambda)$
instead of $\lambda$ and $\beta$.
We have $hv_i=r^i\lambda v_i$
and $Hv_i=s^i\mu v_i$, hence
$ud v_i=(H-\cp(h))v_i=[s^i\mu-\cp(r^i\lambda)]v_i$.
The module $W(\lambda,\mu)$ is
simple if the weights $(r^i\lambda,s^i\mu)$
are distinct and $s^i\mu-\cp(r^i\lambda)\neq 0$
for all $i\in\Z$.

Suppose now that $o(s)=\infty$.
Set $\lambda=0$. Then the weights
$(r^i\lambda,s^i\mu)=(0,s^i\mu)$
are certainly distinct if $\mu\neq 0$, and we can
choose a nonzero value of $\mu$ such that
$s^i\mu-\cp(0)\neq 0$ for all $i\in\Z$.
Thus $W(0,\mu)$ is simple;
its annihilator is $\langle h\rangle$
by Lemma~\ref{notlarger}. This proves
(1).

Suppose next that $o(r)=\infty$
and $\cp$ is not identically zero.
Set $\mu=0$. The weights 
$(r^i\lambda,s^i\mu)=(r^i\lambda,0)$
are distinct if $\lambda\neq 0$.
In this case, $s^i\mu-\cp(r^i\lambda)$
is equal to $-\cp(r^i\lambda)$.
Since $\cp$ is not identically zero,
there are only finitely many
values $x$ such that $\cp(x)=0$;
thus there are only countably
many values of $\lambda$
such that $\cp(r^i\lambda)=0$
for some $i\in\Z$. If we choose
a nonzero value of $\lambda$
outside these countably many
special values, then 
$\cp(r^i\lambda)\neq 0$ for 
\emph{all} $i\in\Z$. For this
value of $\lambda$, $W(\lambda,0)$
is simple. Its annihilator must
be $\langle H\rangle$ by 
Lemma~\ref{confbigideals}.
This proves (2).

Suppose now that $\Se(r,s)=\{(n,0)\}$. 
Then $s$ is not a root of
unity, so the weights $(r^i\lambda,s^i\mu)$
are distinct when $\mu\neq 0$. 
Now if $n\neq 0$, then for any $c\in\C^\times$,
choose $\lambda$
so that $\lambda^n=1/c$. 
If $n=0$, we can choose any nonzero
value for $\lambda$; in this
case $c$ has to be 1. 
Then $h^n$ acts as the scalar $\lambda^n
=c$.
Furthermore,
there are only countably many
values of $s^{-i}\cp(r^i\lambda)$ as
$i$ ranges over $\Z$,
so if we choose $\mu$ to be
different from these countably
many values, then we have
$s^i\mu-\cp(r^i\lambda)\neq 0$
for all $i\in\Z$. Then $W(\lambda,\mu)$
is simple, and its annihilator must
be $I_c$ (or $I_1$ when $n=m=0$)
by Lemmas~\ref{containspower2}
and~\ref{bigidealimplyfindim}.
This proves (1) and parts of (6).

Suppose next  that $\Se(r,s)=\langle (n,-m)\rangle$. 
Then $h^nH^mv_i=r^ns^m\lambda^n\mu^m
v_i=\lambda^n\mu^m v_i$, i.e., 
$h^nH^m$ acts as the scalar $\lambda^n\mu^m$
on $W(\lambda,\mu)$. I claim that
for any $c\in\C^\times$, we can find
values for $\lambda$ and $\mu$ 
such that $\lambda^n\mu^m=c$,
$(r^i\lambda,s^i\mu)$ are all distinct,
and $s^i\mu-\cp(r^i\lambda)\neq 0$
for all $i\in\Z$. This would imply that
$W(\lambda,\mu)$ is simple. Its annihilator
contains $I_c$, and by Lemma~\ref{bigidealimplyfindim}
it cannot be larger than $I_c$.

To prove the claim, set $\mu^m=
c/\lambda^n$. Then the equation
$s^i \mu -\cp(r^i\lambda)=0$
implies that  $s^{im}c=\lambda^n\cp(r^i\lambda)^m$.
If $\cp$ is the zero polynomial,
then there is no solution for any $i\in\Z$.
Otherwise, for each $i\in\Z$, there
are only finitely many values 
of $\lambda$ satisfying the equation
(recall that $n>0$, so 
$\lambda^n\cp(r^i\lambda)^m$ is
a nontrivial polynomial in $\lambda$).
Thus there are at most countably
many values of $\lambda$ such that
$s^{im}c=\lambda^n\cp(r^i\lambda)^m$
for some $i\in\Z$. If we choose
a nonzero value of $\lambda$ different from
any of these countably many values,
we get $s^{im}c/\lambda^n \neq \cp(r^i\lambda)^m$
for all $i\in\Z$, and hence
$s^i\mu-\cp(r^i\lambda)\neq 0$ for
all $i\in\Z$. This proves the claim and
finishes the proof of (5). 

Suppose next that $\Se(r,s)=\langle (n,m)\rangle$
with $m\neq 0$.
Then $H^mv_i=s^m\mu^m v_i
=r^n\mu^m v_i=(\mu^m/\lambda^n)r^n\lambda^n v_i
=(\mu^m/\lambda^n) h^nv_i$, i.e., 
$H^m-(\mu^m/\lambda^n)h^n$ is in
$\ann W(\lambda,\mu)$. I now claim that,
except in the case of Lemma~\ref{specialcase2},
for any $c\in\C^\times$ 
we can always find values of $\lambda$
and $\mu$ such that $c=\mu^m/\lambda^n$,
$(r^i\lambda,s^i\mu)$ are all distinct,
and $s^i\mu-\cp(r^i\lambda)\neq 0$
for all $i\in\Z$. As before, this would imply that
$W(\lambda,\mu)$ is simple; its annihilator
contains $I_c$, and by Lemma~\ref{bigidealimplyfindim}
it cannot be larger than $I_c$.

To prove the claim, we set $\mu^m=c\lambda^n$.
Then the equation 
$s^i\mu-\cp(r^i\lambda)= 0$ implies
that
$s^{im}c\lambda^n = \cp(r^i\lambda)^m$.
This is a nontrivial polynomial equation in
$\lambda$, unless $\cp(h)=Ch^j$ for some
$j\geq 0$, $n=jm$, and $r^{ijm}C^m=s^{im}c$
(and note that the last equality implies
$C^m=c$ since $r^{jm}=s^m$). Thus outside
the special case of Lemma~\ref{specialcase2},
there are at most countably many values
of $\lambda$ such that $s^i\mu-\cp(r^i\lambda)= 0$
for some $i\in\Z$. All we have to do
now is pick a nonzero value of $\lambda$
outside this countably many set of values;
as before, this suffices to prove the claim. 
We have now proven (6). 

We now tackle the special case of
Lemma~\ref{specialcase2}. Recall
that in this case we have $r^n=s^m$,
$n=jm$, $\cp(h)=Ch^j$, and $o(r^js^{-1})=m$.
Write $\theta$ for $r^js^{-1}$. We construct
an infinite-dimensional simple module 
whose annihilator is $\langle d^m\rangle$;
we can show that $\langle u^m\rangle$
is primitive in a similar fashion.

The construction proceeds as follows. (Note
that it works even when $j=0$.)
The module $M$ has basis $\{v_i:i\in\Z\}$,
and the action of $L$ on $M$ is given by
\[
uv_i=v_{i+1},\quad
dv_i=Cs^i(1-\theta^i)v_{i-n-1},\quad
hv_i=r^{i+(n-m)/2}v_{i-m},
\]
note that we might need a square root of $r$.
It is not hard to check that $huv_i=ruhv_i$
for all $i\in\Z$. Verifying the other relations
is not as straightforward. We have
\[
dhv_i=r^{i+(n-m)/2}dv_{i-m}
=Cr^{i+(n-m)/2}s^{i-m}(1-\theta^{i-m})v_{i-m-n-1},
\]
while
\[
rhdv_i=rCs^i(1-\theta^i)hv_{i-n-1}
=Cr^{i-n+(n-m)/2}s^i(1-\theta^i)v_{i-n-1-m}.
\]
Since $\theta^m=1$ and $r^n=s^m$, the two
expressions are equal, proving that $dhv_i=rhdv_i$
for all $i\in\Z$. 

To verify the last relation, we first calculate
that $h^jv_i= r^{ij}v_{i-n}$. In the algebra
$L$, we have $du-sud=(s-r^j)Ch^j$. Now
\[
(du-sud)v_i=Cs^{i+1}[(1-\theta^{i+1})-(1-\theta^i)]v_{i-n}
=Cs^{i+1}\theta^i(1-\theta)v_{i-n},
\]
while
\[
(s-r^j)Ch^jv_i=(s-r^j)Cr^{ij}v_{i-n}=C(s-s\theta)(s\theta)^iv_{i-n}
=Cs^{i+1}\theta^i(1-\theta)v_{i-n},
\]
thus verifying that $M$ is indeed an $L$-module.

It is easy to show that $M$ is simple. Each $v_i$ is
an eigenvector of $u^mh$ with eigenvalue
$r^{(n-m)/2} r^i$. These eigenvalues
are distinct (recall that $r$ is not a root of unity),
implying that $M$ is simple. Note that 
$d^m\in\ann M$; by Lemma~\ref{bigidealimplyfindim}
we must have $\ann M=\langle d^m\rangle$.
This proves (7).

We now note that the construction above works
when $C=0$ (hence $\cp=0$) and $n=m=1$. The action
of $L$ becomes
\[
uv_i=v_{i+1},\quad
dv_i=0,\quad
hv_i=r^{i}v_{i-1}.
\]
It is straightforward to check that with this
action of $L$, $M$ is still
an $L$-module no matter the values
of $r$ and $s$, as long as $\cp$ is
identically zero. When $r$ is not a root
of unity, the proof above shows that 
$M$ is simple. Also as above, its annihilator
is $\langle d\rangle$. Thus $\langle d\rangle$
is primitive. Similarly, $\langle u\rangle$
is primitive. This proves (3). 
\end{proof}

\subsection{A list of primitive ideals}\label{conformallist}
We now summarize our results and provide
a list of primitive ideals of $L$ when $L$ is conformal.
As before we include only primitive ideals
that are annihilators of infinite-dimensional
modules. In the first table below, we list 
the primitive ideals when $r$ or $s$
is a root of unity. 

\bigskip
\begin{center}
\begin{tabular}{|c|c|c|c|}\hline
$o(r)$ & $o(s)$ & $\cp$ & primitive ideals \\ \hline
$n$ & $m$ & any & $-$ \\
$n$ & $\infty$ & any & $\langle h\rangle$, 
	$\langle h^n-c\rangle$ ($c\in\C^\times$)\\
$\infty$ & $m$ & $0$ & $\langle u\rangle$,
	$\langle d\rangle$, $\langle H^m-c\rangle$
	($c\in\C^\times$)\\
$\infty$ & $m$ & $C\neq 0$ & $\langle u^m\rangle$,
	$\langle d^m\rangle$, $\langle H\rangle$, 
	$\langle H^m-c^m\rangle$ ($c\in\C^\times,
	c^m\neq C^m$)\\
$\infty$ & $m$ & nonconstant & $\langle H\rangle$,
	$\langle H^m-c\rangle$ ($c\in\C^\times$)\\
\hline
\end{tabular}
\end{center}
\bigskip

In the second table, we list primitive ideals
when both $r$ and $s$ are not roots of unity.
The results are expressed in terms of
the generator of $\Se(r,s)$. In the table, 
$n$, $m$, and $j$ are all positive integers.

\bigskip
\begin{center}
\begin{tabular}{|c|c|c|}\hline
generator & $\cp$ & primitive ideals \\ \hline
$(0,0)$ & $0$ & $\{0\}$, $\langle h\rangle$,
	$\langle u\rangle$, $\langle d\rangle$\\
$(0,0)$ & nonzero & $\{0\}$, $\langle h\rangle$, 
	$\langle H\rangle$\\
$(n,-m)$ & $0$ & $\langle h\rangle$,
	$\langle u\rangle$, $\langle d\rangle$,
	$\langle h^nH^m-c\rangle$ ($c\in\C^\times$)\\
$(n,-m)$ & nonzero & $\langle h\rangle$,
	$\langle H\rangle$,
	$\langle h^nH^m-c\rangle$ ($c\in\C^\times$)\\
$(jm,m)$ & $Ch^j$ ($C\neq 0$) & $\langle h\rangle$,
	$\langle H\rangle$, $\langle u^m\rangle$,
	$\langle d^m\rangle$,
	$\langle H^m-ch^{jm}\rangle$ ($c\in\C^\times, c\neq C^m$)\\
$(jm,m)$ & $0$ & $\langle h\rangle$,
	$\langle u\rangle$, $\langle d\rangle$,
	$\langle H^m-ch^{jm}\rangle$ ($c\in\C^\times$)\\
$(jm,m)$ & $\neq 0,\neq Ch^j$ & $\langle h\rangle$,
	$\langle H\rangle$, 
	$\langle H^m-ch^{jm}\rangle$ ($c\in\C^\times$)\\
$(n,m)$, $n\neq jm$ & $0$ & $\langle h\rangle$,
	$\langle u\rangle$, $\langle d\rangle$,
	$\langle H^m-ch^{n}\rangle$ ($c\in\C^\times$)\\
$(n,m)$, $n\neq jm$ & nonzero & $\langle h\rangle$,
	$\langle H\rangle$, 
	$\langle H^m-ch^{n}\rangle$ ($c\in\C^\times$)\\
\hline
\end{tabular}
\end{center}

\section{When $L$ is not conformal, $\gamma=0$}

In this section we assume that  $L$ is not conformal
and $\gamma=0$. If $\phi(h)=\sum_i a_ih^i$, recall
our assumption means $s=r^j$ and $a_j\neq 0$ 
for some $j\geq 0$. If  $r$ is not a root of
unity, then there is only one such value 
of $j$, but if $r$ is a root of unity, then
there could be more than one value
of $j$. 

We now establish some notation. Suppose
first that $r$
is a root of unity, with $o(r)=n$. We
fix $j$ so that $s=r^j$ and $0\leq j\leq n-1$.
The polynomial $\phi$ can then be
written as a sum of a conformal part $\phi_0$
and a nonconformal part $\phi_1$, where
\[
\phi_0(h)
=\sum_{i\not\equiv j} a_ih^i, \quad
\phi_1(h)=\sum_{i\equiv j} a_ih^i,
\]
(the congruences are mod $n$). Note
that $\phi_0$ is conformal, while
$\phi_1$ can be written as
$\phi_1(h)=sh^j\ncp(h^n)$ for some
polynomial $\ncp$. (We use the
constant factor $s$ in order to make
subsequent formulas easier to work 
with.) If $\cp$ is a conformal 
polynomial corresponding to $\phi_0$,
then $\phi(h)=s\cp(h)-\cp(rh)+sh^j\ncp(h^n)$.

The case $n=1$ (i.e., $r=s=1$) perhaps deserves
a special mention. Here $\phi_0(h)$
is the zero polynomial, and $\ncp(h^n)
=\ncp(h)=s^{-1}\phi(h)$. Thus we
can take $\cp(h)$ to be zero, and
all the action takes place in the
nonconformal part.

The situation is a bit simpler when
$r$ is not a root of unity.
We still split $\phi$ into a conformal
part and a nonconformal part, but
the nonconformal part is just a monomial:
$\phi_1(h)=sCh^j$ for some $C\neq 0$.

Note that this second case can be subsumed
into the previous case when 
$\ncp(h)=C$. Thus we can include
the second case in the first, even
though $h^n$ doesn't make sense
when $r$ and $s$ are not roots of
unity.  

Indeed, the two cases can be
made even more similar. Suppose
again that $o(r)=n$. 
Since $hu=ruh$, we have
$h^nu=r^nuh^n=uh^n$. Similarly,
$dh^n=h^nd$. Thus $h^n$ is a central
element, so $\ncp(h^n)$ also commutes
with everything in $L$. It follows
that if $M$ is a simple module,
then $\ncp(h^n)$ acts as a
scalar $C=C(M)$. If $C=0$,
then the nonconformal part
$\phi_1(h)$ acts as the zero
operator on $M$, and thus we
are in the same situation as the
conformal case. Looking at 
the first line of the first table in
section~\ref{conformallist},
we conclude that in this case
$M$ must be finite-dimensional.
Therefore
we can assume from now
on that the scalar $C$ is nonzero;
we can then write $\phi(h)=
s\cp(h)-\cp(rh)+sCh^j$ as operators
on $M$. This is just like
the case where $o(r)=\infty$. 

In either case, just as in the 
conformal case, we define $H$ to be
$ud+\cp(h)$. A straightforward
calculation then shows that
$Hu=su(H+\ncp(h^n)h^j)$
and $dH=s(H+\ncp(h^n)h^j)d$.
Then 
$Hx_i=s^ix_i(H+i\ncp(h^n) h^j)$
for $x_i\in L_i$. We can replace
$\ncp(h^n)$ by the nonzero
constant $C$ if the terms
are viewed as operators on
a simple module $M$.

\subsection{Homogeneous
elements in ideals}
We have the following lemma about
homogeneous elements in ideals, the
analogue of Lemmas~\ref{homogen}
and~\ref{homogen2}. It is unavoidably
more complicated in the
case $o(r)=n$. 
\begin{lem}\label{homogen3}
Let $I$ be an ideal of $L$. If $o(r)=n$,
assume also that $\ncp(h^n)-C\in I$, 
where $C\neq 0$.
Suppose $x\in I$ and $x=\sum_{i\in\Z} x_i$ (where
$x_i\in L_i$) is the homogeneous decomposition
of $x$. Then there exists an integer $m\geq 0$
such that $x_i h^m\in I$  
for all $i\in\Z$. 

\end{lem}
\begin{proof}
If $o(r)=\infty$, this is simply
Lemma~\ref{homogen2}. Suppose
now that $o(r)=n$. 
We again use induction on the
length of $x$. For a given $k\in\Z$ with
$x_k\neq 0$, we
have
\[
hx-r^kxh=\sum_{i\in\Z} (r^i-r^k)x_ih.
\]
The element $hx-r^kxh$ thus
has smaller length than $x$.
Then for all $i\in\Z$
such that $r^i\neq r^k$,
we have $x_i h^m\in I$ for some 
$m\geq 0$. Therefore the element
$y=\sum_{\{i: r^i=r^k\}} x_i h^m$
is also in $I$. 

It follows that $Hy=\sum_{\{i: r^i=r^k\}}
s^ix_i(H+i\ncp(h^n)h^j)h^m$ is also in $I$.
Since $\ncp(h^n)-C$ is an element
of $I$, we conclude that 
$z=\sum_{\{i: r^i=r^k\}}
s^ix_i(H+iC h^j)h^m\in I$. 
Now recall  that $r^i=r^k$
implies that $s^i=s^k$. Thus
\begin{align*}
z&-s^kyH-ks^kCy h^j\\
&=\sum_{\{i: r^i=r^k\}}
[s^ix_i H h^m +is^iCx_ih^{j+m}] - \sum_{\{i: r^i=r^k\}}
[s^kx_i h^m H+ks^k C x_ih^{m+j}]\\
&=\sum_{\{i: r^i=r^k\}} s^k(i-k)Cx_i h^{m+j}.
\end{align*}
This element has smaller length than $y$,
so by the induction hypothesis, for each
$i\neq k$, the homogeneous element
$x_i h^{m'}$ is in $I$. It follows
that $x_k h^{m'}$ is also
in $I$.
\end{proof}

Now let $M$ be a simple module. 
As before, either $hM=0$ or $hM=M$.
We first look at the situation
where $hM=0$.

\subsection{The case $hM=0$}\label{sectionhM=0}
In this case $h$ acts as the zero operator
on $M$. Recall that the nonconformal part 
of $\phi(h)$ can be written as
 $\phi_1(h)=sC h^j$, where $C\neq 0$.
Thus if $j>0$, 
then $\phi_1(h)$ also
acts as the zero operator on $M$.  In this case
we have
exactly the same situation as the conformal
case: if $s$ is not a root of unity,
then $\ann M$ cannot be larger
than $\langle h\rangle$, but if
$s$ is a root of unity, then $M$ must
be finite-dimensional.

Suppose now that $j=0$. Then 
$s=1$ and $\phi(0)=C$; 
also, $Hu=u(H+C)$
and $dH=(H+C)d$. These relations
are exactly analogous to the relations
between $h$ and $u$ (and between
$h$ and $d$) in
the case $r=1, \gamma\neq 0$,
with $H$ taking the role of $h$
and $C$ taking the role of $\gamma$.
In particular, Lemmas~\ref{homogen}
and~\ref{containspower}
are valid; that is,
if $\ann M\supsetneq \langle h
\rangle$, then $\ann M$ contains
a power of $u$ or a power of $d$.
Let's say $\ann M$ contains $d^k$
but not $d^{k-1}$
for some $k> 0$. Since
$du=ud+C$ as operators on $M$,
we have
\[
d^ku=ud^k+kCd^{k-1};
\]
this is a contradiction since it implies
that $Cd^{k-1}$ annihilates $M$.
Thus in this case, $\ann M$ cannot
be larger than just $\langle h\rangle$. 

\subsection{The case $hM=M$}
If $o(r)=n$, recall that $h^n$ acts as a scalar
on $M$. Thus any primitive ideal must contain $\langle
h^n-c\rangle$ for some $c\in\C$. We are
assuming that $\ncp(c)=C\neq 0$;
since $hM=M$, it is also true that $c\neq 0$. (On the other
hand, if $o(r)=\infty$, we have no obvious information
on what a primitive ideal must contain.)

We need the following result on polynomials of two
variables. It is a version of Lemmas~\ref{polyproperty}
and~\ref{distinctpolyproperty}, but
unfortunately it is rather more complicated.
\begin{lem}\label{ultimatepolyproperty}
Suppose $r^j=s$ and $f(x,y)=\sum_{i=0}^b 
g_{i}(x)y^{i}$
is a two-variable polynomial of
degree $(a,b)$, where $\text{deg }g_i(x)
< o(r)$ for $0\leq i\leq b$.  If 
\[
f(rx,sCx^j+sy) = r^as^b f(x,y)
\]
for some $C\neq 0$,
then $f(x,y)$ is a constant
multiple of $x^a$.
\end{lem}
\begin{proof}
We first want to show that $b=0$.
Assume for a contradiction that $b>0$.
Then we can write
\[
f(x,y)=g_b(x)y^b + g_{b-1}(x)y^{b-1}+\cdots
\]
where $g_b(x)$ has degree $a$.
Then
\begin{align*}
f(rx,sCx^j+sy)&=g_b(rx)(sCx^j+sy)^b +
g_{b-1}(rx)(sCx^j+sy)^{b-1} + \cdots\\
&=s^b g_b(rx) y^b + s^b  g_b(rx) b Cx^j y^{b-1}
+ s^{b-1} g_{b-1}(rx) y^{b-1} + \cdots\\
&= s^b g_b(rx) y^b + [bCs^b x^j g_b(rx) +
s^{b-1} g_{b-1}(rx)]y^{b-1} + \cdots
\end{align*}
This is equal to
\[
r^{a}s^b g_b(x) y^b + r^{a}s^b g_{b-1}(x) y^{b-1} +\cdots
\]
Equating the coefficients of $y^b$
and of $y^{b-1}$, we get 
\begin{gather*}
g_b(rx) = r^a g_b(x),\\
bCs^b x^j g_b(rx) + s^{b-1} g_{b-1}(rx) = r^{a}s^b g_{b-1}(x).
\end{gather*}
Simplifying the second equation gives us
\[
b C r^{a+j} x^j g_b(x) = r^{a+j} g_{b-1}(x) - g_{b-1}(rx).
\]
The left-hand side is a polynomial
of degree $a+j$ since $g_b$ has degree $a$.
Now look at the right-hand side.
If $g_{b-1}(x) = \sum c_i x^i$, then
\[
r^{a+j} g_{b-1}(x) - g_{b-1}(rx) = \sum c_i(r^{a+j}-r^i)x^i,
\]
so the coefficient of $x^{a+j}$ with is 0. 
We have the required contradiction,
thus indeed $b=0$.

We therefore conclude
that $f(x,y)=g(x)$. Suppose
$g(x)=\sum_{i=0}^a c_ix^i$; recall
that $a<o(r)$. We
require that $g(rx)=r^a g(x)$,
but this implies that
$\sum c_i(r^i-r^a)x^i=0$,
thus showing that $c_i=0$
for $i\neq a$. We conclude that
$g(x)=c_ax^a$, as required.
\end{proof}
As usual, we use this lemma to prove that
large primitive ideals contain a power of $u$
or a power of $d$.
\begin{lem}\label{nonconpower}
Suppose $M$ is simple. If $o(r)=n$ and
$\ann M\supsetneq \langle h^n-c\rangle$
(where $c\neq 0$), then $\ann M$ contains
a power of $u$ or a power of $d$.
Similarly, if $o(r)=\infty$ and 
$\ann M\neq \{0\}$, then $\ann M$
contains a power of $u$ or a power 
of $d$.
\end{lem}
\begin{proof}
Write $I$ for the zero ideal if $o(r)=\infty$
and for $\langle h^n-c\rangle$ if
$o(r)=n$. Suppose 
$\ann M\supsetneq I$.
By Lemma~\ref{homogen3},
$\ann M$ contains a homogeneous
element not in $I$.  It does no harm
to assume that
this homogeneous element
has nonnegative degree. Thus $\ann M$
contains elements of the form
$x=u^kf(h,H)$, where $f$ is a nonzero polynomial
that can be written as 
$f(h,H)=\sum_i g_i(h)H^i$ with the degree
of $g_i(x)$ being less than $o(r)$.
Choose
$x=u^kf(h,H)$ to be such an element
where
the degree $(a,b)$ of the polynomial $f$
is as small as possible. 
Then
\[
xu-r^as^bux=u^{k+1}[f(rh,sH+sCh^j)-r^as^bf(h,H)]
\]
is also in $\ann M$, and the
degree of the the polynomial 
$f(rh,sH+sCh^j)-r^as^bf(h,H)$
is smaller than $(a,b)$. Thus it must
be zero. By Lemma~\ref{ultimatepolyproperty},
$f(h,H)$ must be a (nonzero) multiple
of $h^a$. Since $hM=M$, we conclude
that $u^k\in\ann M$, as required.
\end{proof}
\begin{lem}\label{nonconfindim}
Let $M$ be a simple module and
suppose $\ann M$ contains a power of
$u$ or a power of $d$. Then $M$
is finite-dimensional.
\end{lem}
\begin{proof}
Let's say $\ann M$ contains $d^k$ but
not $d^{k-1}$. Since $du=sud+\phi_0(h)+sCh^j$,
we have
\[
d^ku-s^kud^k = 
\left[\sum_{i=0}^{k-1} s^i\phi_0(r^{k-i-1}h)+ks^kCh^j\right]d^{k-1}.
\]
The polynomial in square brackets is
nonzero since none of the monomials in
$\phi_0$ can cancel out $h^j$. By
Corollary~\ref{findimcor}, $M$ is
finite-dimensional.
\end{proof}
It remains to show that $\{0\}$ is 
a primitive ideal if $o(r)=\infty$; also
that $\langle h^n-c\rangle$ is
a primitive ideal when $o(r)=n$.
\begin{lem}
Suppose $L$ is not conformal, $s=r^j$,
and $r$ is not a root of unity. Then 
$\{0\}$ is a primitive ideal.

On the other hand, suppose
$o(r)=n$ and $c\in\C^\times$ satisfies
$\ncp(c)\neq 0$. Then the ideal
$\langle h^n-c\rangle$ is a primitive
ideal.

Suppose further that 
$j=0$ and $\ncp(0)\neq 0$. Then 
$\langle h\rangle$ is a primitive
ideal.
\end{lem}
\begin{proof}
As usual, we use the universal weight
modules $W(\lambda,\beta)$. As before,
it is more convenient to use the
parametrization $(\lambda,\mu)$
where $\mu=\beta+\cp(\lambda)$
is the $H$-eigenvalue of the
vector $v_0$. Since $Hu=su(H+Ch^j)$
(where $C=\ncp(c)$ if $o(r)=n$;
if $o(r)=\infty$, $C$ is simply
the coefficient of $h^j$ in $\phi_1(h)$),
we can figure out that
$Hv_i=s^i(\mu+iC\lambda^i)v_i$
for all $i\in\Z$. Thus
\[
ud v_i= [s^i(\mu+iC\lambda^i)-\cp(r^i\lambda)]v_i.
\]
Note that each $v_i$ is an
eigenvector of $H$ with distinct
eigenvalues; thus $W(\lambda,\mu)$
is simple if $\beta_i=
s^i(\mu+iC\lambda^i)-\cp(r^i\lambda)$
is nonzero for all $i\in\Z$.

Now for any given value of
$\lambda$, there are only
countably many values of
$s^{-i}\cp(r^i\lambda)-iC\lambda^i$
as $i$ ranges over $\Z$.
If we pick $\mu$ outside
these values, then $\beta_i
\neq 0$ for all $i\in\Z$
and $W(\lambda,\mu)$ would
be simple. We now see what
happens for various values of
$\lambda$.

If $\lambda=0$, then $h$
acts as the zero operator
on $W(\lambda,\mu)$. If in
addition $j=0$ and $\ncp(0)\neq 0$,
then the
annihilator of $W(\lambda,\mu)$
cannot be larger than 
$\langle h\rangle$, as discussed
in section~\ref{sectionhM=0}.
Thus $\langle h\rangle$ is
indeed a primitive ideal
in this situation. 

If $\lambda\neq 0$, then
$h$ does not act as the zero
operator, hence $hW(\lambda,\mu)=
W(\lambda,\mu)$. If in
addition $r$ is not a
root of unity, then the
annihilator of $W(\lambda,\mu)$
cannot be larger
than $\{0\}$, by
Lemmas~\ref{nonconpower}
and~\ref{nonconfindim}, so
$\{0\}$ is indeed a primitive
ideal.

If $o(r)=n$, then for any
value $c\in \C^\times$ such
that $\ncp(c)\neq 0$,
we can pick $\lambda\in\C$
with $\lambda^n=c$ (note that
$\lambda\neq 0$); after
that, we pick $\mu$ as above
to make $W(\lambda,\mu)$
a simple module. 
Then $h^n$ acts as the
scalar $c$ on $W(\lambda,\mu)$.
The annihilator of $W(\lambda,\mu)$
is just $\langle h^n-c\rangle$
by Lemmas~\ref{nonconpower}
and~\ref{nonconfindim}.
\end{proof}

\subsection{A list of primitive ideals}
We can put all these results together
to generate an explicit list of the primitive
ideals of $L$ when $L$ is not conformal.
Recall that $s=r^j$. If $o(r)=n$, we 
write $\phi(h)=\phi_0(h)+s\ncp(h^n)h^j$
where $\phi_0(h)=\sum_{i\not\equiv j}
a_ih^i$. As before, we list primitive ideals
that are annihilators of infinite-dimensional
modules.

\bigskip
\begin{center}
\begin{tabular}{|c|c|c|c|}\hline
$j$ & $o(r)$ & $p(0)$ & primitive ideals\\ \hline
any & $\infty$ & n/a & $\{0\}, \langle h\rangle$ \\
$0$ & $n$ & $0$ & $\langle h^n-c\rangle$, $c\in\C^\times,
	\ncp(c)\neq 0$\\
$0$ & $n$ & nonzero & $\langle h\rangle$, 
	$\langle h^n-c\rangle$, $c\in\C^\times,
	\ncp(c)\neq 0$\\
$>0$ & $n$ & any & $\langle h^n-c\rangle$, $c\in\C^\times,
	\ncp(c)\neq 0$\\
\hline
\end{tabular}
\end{center}

\noindent
Department of Mathematics\\
Franklin and Marshall College, Lancaster, PA 17604\\
\texttt{iwan.praton@fandm.edu}

\end{document}